\documentclass[11pt,reqno]{amsart}
\usepackage{amssymb,dsfont,enumerate}
\usepackage[pagebackref]{hyperref}
\usepackage[portrait,a4paper,margin=3cm]{geometry}

\usepackage{mathtools}
\mathtoolsset{showonlyrefs}

\numberwithin{equation}{section}

\date{Summer 2016, compiled \today} 

\author{Pietro Caputo}
\address[P.~Caputo]{Dipartimento di Matematica,
  Universit\`a Roma Tre, Italy}
\email{caputo(at)mat.uniroma3.it}
\urladdr{http://www.mat.uniroma3.it/users/caputo/}

\author{Alistair Sinclair}
\address[A.~Sinclair]{Computer Science Division, 
Soda Hall,  
University of California 
Berkeley, CA 94720-1776, USA}
\email{sinclair(at)cs.berkeley.edu}
\urladdr{https://www.cs.berkeley.edu/~sinclair/}

\DeclareMathSymbol{\leqslant}{\mathalpha}{AMSa}{"36} 
\DeclareMathSymbol{\geqslant}{\mathalpha}{AMSa}{"3E} 
\DeclareMathSymbol{\eset}{\mathalpha}{AMSb}{"3F}     
\renewcommand{\leq}{\;\leqslant\;}                   
\renewcommand{\geq}{\;\geqslant\;}                   
\newcommand{\sumtwo}[2]{\sum_{\substack{#1 \\ #2}}} 

\newcommand{\be}{\begin{equation}}

\def\thsp{\thinspace}
\def\tc{\thsp | \thsp}
\def\1{\ifmmode {1\hskip -3pt \rm{I}} \else {\hbox {$1\hskip -3pt \rm{I}$}}\fi}

\newcommand{\ind}{\mathbf{1}}
\newcommand\at[2]{\left.#1\right|_{#2}}

\newtheorem{theorem}{Theorem}[section] 
\newtheorem{lemma}[theorem]{Lemma} 
\newtheorem{proposition}[theorem]{Proposition} 
\newtheorem{corollary}[theorem]{Corollary} 
\newtheorem{remark}[theorem]{Remark} 
\newtheorem{conjecture}[theorem]{Conjecture} 
\newtheorem{definition}[theorem]{Definition}

\newcommand{\cA}{\ensuremath{\mathcal A}}

\newcommand{\cF}{\ensuremath{\mathcal F}} 
\newcommand{\cG}{\ensuremath{\mathcal G}}

\newcommand{\cP}{\ensuremath{\mathcal P}}

\newcommand{\cS}{\ensuremath{\mathcal S}}

\newcommand{\cX}{\ensuremath{\mathcal X}} 
 
\newcommand{\cZ}{\ensuremath{\mathcal Z}} 
 
\newcommand{\bbE}{{\ensuremath{\mathbb E}} }

\newcommand{\bbN}{{\ensuremath{\mathbb N}} }

\newcommand{\bbR}{{\ensuremath{\mathbb R}} }

\newcommand{\si}{\sigma} 
\newcommand{\wt}{\widetilde} 
\newcommand{\ent}{{\rm Ent} } 
\newcommand{\var}{{\rm Var} } 
 
\newcommand{\dert}{\frac{{\rm d}}{{\rm d} t}}

\let\a=\alpha \let\b=\beta   \let\d=\delta  \let\e=\varepsilon
 \let\g=\gamma     \let\k=\kappa  
            
\let\r=\rho   \let\t=\tau   
  
\let\D=\Delta   \let\G=\Gamma   
\let\O=\Omega      

\begin{document}

\title[Entropy production in nonlinear recombination models]
{Entropy production in nonlinear \\ recombination models}

\begin{abstract}
We study the convergence to equilibrium of a class of nonlinear recombination models. In analogy with Boltzmann's H theorem from kinetic theory, and in contrast with previous analysis of these models, convergence is measured in terms of relative entropy. The problem is formulated within a general framework that we refer to as Reversible Quadratic Systems. Our main result is a tight quantitative estimate for the entropy production functional. Along the way we establish some new entropy inequalities generalizing Shearer's and related inequalities.     
\end{abstract}

\thanks{This work was done while P.C.\ was visiting the Simons Institute for the Theory of Computing, supported by the Simons Foundation.
A.S.\ was supported in part by NSF grant CCF-1420934.}

                             
\maketitle

\thispagestyle{empty}

\section{Introduction}
Recombination models based on random mating have a wide range of applications in the natural sciences and play an important role in the analysis of genetic algorithms.  The following nonlinear system is a commonly studied model. Let $\O=X_1\times\cdots\times X_n$ denote the set of sequences of length $n$, such that the $i$-th element of the sequence takes values in a given finite space $X_i$. A sequence  $\si\in\O$ is written as $\si=(\si_i,\,i\in[n])$, where $[n]=\{1,\dots,n\}$ is the set of {\it loci}, or {\it sites}, and $\si_i\in X_i$ for all $i\in[n]$.  Given a subset $A\subset[n]$, and a $\si\in\O$ we write $\si_A$ for the $A$-component of $\si$, i.e., the subsequence $(\si_i,\,i\in A)$. If $(\si,\eta)\in\O\times \O$ is a pair of sequences, the {\it recombination\/} at $A$ consists in exchanging the $A$-component of $\si$ with the $A$-component of $\eta$. This defines the map 
$$
(\si,\eta)\mapsto (\eta_A\si_{A^c},\si_A\eta_{A^c})\,.
$$ 
If the original pair $(\si,\eta)$ is obtained by sampling independently from a probability measure $p$ on $\O$, then the sequence $\eta_A\si_{A^c}$ is distributed according to  $p_A\otimes p_{A^c}$, the product measure obtained from the marginals of $p$ on $A$ and $A^c$. By choosing the set $A$ at random according to some probability distribution $\nu$ one obtains the quadratic dynamical system 
\begin{equation}\label{ptilde}
p\mapsto \Psi[p] :=
 \sum_{A} \nu(A) \,(p_A\otimes p_{A^c}).
\end{equation}
 The discrete time evolution of the initial distribution $p$ is then defined by iteration of the map $\Psi$, namely $p^{(k)}= \Psi[p^{(k-1)}]$, $k\in\bbN$, $p^{(0)}=p$. Analogously, in continuous time one has the quadratic differential equation
 \begin{equation}\label{pct}
\dert \,p_t=
 \sum_{A} \nu(A) \,(p_{t,A}\otimes p_{t,A^c} - p_t)\,,\quad t\geq 0,
\end{equation}
with the initial condition $p_0=p$. Here $p_t$ is the probability measure describing the state at time $t$, and $p_{t,A}$ denotes its marginal on $A$. When $A=\{i\}$ is a single site we write $p_{t,i}$ for the marginal at $i$. It is not difficult to see that the map \eqref{ptilde} preserves the single site marginals, so that $p_{t,i}=p_{0,i}$ for all $t\geq 0$ and for all $i\in[n]$.  
The study of this model goes back to the pioneering work of Geiringer \cite{Gei}; 
see also \cite{Sinetal,Sinetal2,Baakeetal,Martinez} for more recent accounts. As emphasized in  \cite{Sinetal,Sinetal2} this model is a special case of the much larger class of so-called ``quadratic dynamical systems", which provides a rich family of discrete analogues of Boltzmann's equation from statistical physics. It is a classical result that, under an obvious nondegeneracy assumption on the distribution $\nu$, the system converges to the stationary state given by the product of the marginals of 
the initial state $p$; i.e., if $p_i=p_{0,i}$ denotes the marginal of $p$ at site $i$, then 
\begin{equation}\label{station}
\pi=\otimes_{i}\,p_i
\end{equation}
is the equilibrium distribution and one has convergence in distribution: $p^{(k)}\to \pi$, $k\to\infty$, and $p_t\to \pi$, $t\to\infty$. We shall be interested in the speed of convergence to equilibrium in both continuous and discrete time. 
We consider the following natural examples of the distribution $\nu$:
\begin{enumerate}[1)]
\item
Single site recombination: $\nu(A)=\frac1n\sum_{i=1}^n\ind(A=\{i\})$;
\item 
One-point crossover: 
$\nu(A)=\tfrac1{n+1}\sum_{i=0}^n\ind(A=J_i)$, where $J_0=\emptyset$, $J_i=\{1,\dots,i\}$, $i\geq 1$;
\item 
Uniform crossover: $\nu(A)=\frac1{2^n}$, for all $A\subset[n]$;
\item The Bernoulli($q$) model: for some $q\in[0,\tfrac12]$, $\nu(A)=q^{|A|}(1-q)^{n-|A|}$.\end{enumerate}
The Bernoulli($q$) model is a generalization of the uniform crossover model. In principle our method can be applied to other generalizations, such as the $k$-crossover model or the so-called Poisson model. However, to keep this work at a reasonable length, in what follows  we will restrict attention to the models listed above.
The first example generates a simple linear evolution, but the other choices produce genuinely nonlinear processes. 
Tight estimates on the speed of convergence in total variation norm of the associated discrete time processes were obtained in \cite{Sinetal2}. While the first example reduces to a standard coupon 
collecting argument, implying that the system mixes in $\Theta(n\log n)$ steps, the other cases require a finer coupling analysis. In particular, it is shown in \cite{Sinetal2} that one-point crossover mixes in $\Theta(n\log n)$ steps, while uniform crossover mixes in $\Theta(\log n)$ steps.
Further results on convergence to equilibrium in total variation norm together with the analysis of the quasi-stationary measure were recently obtained in \cite{Martinez}. 

In this paper we focus on convergence to equilibrium in terms of {\it entropy}.
Here we recall that the {\it relative entropy\/} of $p$ with respect to $\mu$, for two probability measures $p,\mu$ on $\O$, is given by 
$$H(p\tc\mu) =  \sum_{\si\in\O} p(\si)\log (p(\si)/\mu(\si))\,,$$
with  $H(p\tc\mu)=+\infty$ if there exists $\si\in\O$ with $\mu(\si)=0$ and $p(\si)\neq 0$. 
Recall also Pinsker's inequality, asserting that 
 \begin{equation}\label{pinsker}
\|p-\mu\|\leq \sqrt {\tfrac12\, H(p\tc \mu)},
\end{equation}
where $\|p-\mu\|=\tfrac12\sum_{\si\in\O}|p(\si)-\mu(\si)|$ denotes the total variation distance. 
Now, for any probability measure $p$ on $\O$,
the convergence $p_t\to \pi$ implies that the relative entropy $H(p_t\tc\pi)$ satisfies $H(p_t\tc\pi)\to 0$, $t\to\infty$. By analogy with the Boltzmann equation, it is very natural to study the rate of exponential decay of relative entropy or, equivalently, the existence of an inequality of the form
  \begin{equation}\label{entroprod}
\dert H(p_t\tc\pi)\leq - \d\, H(p_t\tc\pi)\,,
\end{equation}
for some $\d>0$ independent of $t$. The bound \eqref{entroprod} is often referred to as an 
{\it entropy production\/} estimate. The investigation of the corresponding inequality for the Boltzmann equation, stimulated by a famous conjecture of Cercignani,  is at the heart of many recent spectacular developments in kinetic theory; see, e.g., \cite{Cercignaniconj} for a survey. To the best of our knowledge the inequality \eqref{entroprod} has not been investigated for the recombination models introduced above. The main purpose of this paper is to initiate this study and provide sharp estimates of the constant $\d$ in these models.  We also discuss possible generalizations and make some preliminary steps towards more general reversible quadratic systems displaying {\it non-product\/} equilibrium measures, such as nonlinear versions of the stochastic Ising model; see Section~\ref{examples} below.  

To describe our main results it is convenient to reformulate the 
entropy production estimate in terms of suitable functional inequalities. Let $\pi$ be a product measure on $\O$ of the form \eqref{station}. For any nonnegative function $f:\O\mapsto[0,\infty)$, define the entropy functional
\begin{equation}\label{entrop}
\ent(f)=\pi[f\log f] -\pi[f]\log \pi[f].
\end{equation}
For any $A\subset[n]$, let $f_A$ 
denote the function
\begin{equation}\label{fa}
f_A(\si)=\sum_{\eta\in\O}\pi(\eta)f(\si_A\eta_{A^c});
\end{equation}
note that $f_A$ depends only on~$\si_A$. 
When $A=\{i\}$ for some $i\in[n]$, we simply write $f_i$ for $f_{\{i\}}$. 
Let $\cS_\pi$ denote the set of all $f:\O\mapsto[0,\infty)$ such that $\pi[f]=1$ and $f_i=1$ for all $i\in[n]$. Notice that $\cS_\pi$ is precisely the set of all functions $f$ of the form $f=p/\pi$, where $p$ is any probability measure on $\O$ satisfying~\eqref{station}; i.e., $f$ is the density of~$p$ with respect
to~$\pi$ and $p$ has the same marginals as $\pi$.  Moreover, $f_Af_{A^c}$ is the density of $p_A\otimes p_{A^c}$ with respect to $\pi$. 
Given a distribution $\nu$ over subsets of $[n]$ we call $\d(\pi,\nu)$ the largest constant $\d\geq 0$ such that 
the inequality
\begin{equation}\label{bord}
 \sum_{A} \nu(A)\,\pi\left[(f_A f_{A^c}-f)\log {\frac{f_A f_{A^c}}f}\right]\geq \d\,\ent(f)
\end{equation}
holds for all $f\in\cS_\pi$. As we shall see, inequality \eqref{bord} coincides with \eqref{entroprod} when $f=p_t/\pi$ is the density of $p_t$ with respect to $\pi$. Define also 
\begin{equation}\label{bord2}
\d(\nu)=\inf_\pi \d(\pi,\nu)\,,
\end{equation}
where the infimum is taken over all product measures on $X_1\times\cdots\times X_n$ and over all possible underlying finite spaces $X_i$. 
\begin{theorem}\label{main}
The recombination models defined above satisfy the following bounds:
\begin{enumerate}[1)]
\item
Single site recombination: $\tfrac2n +O(n^{-2}) 
\geq \d(\nu)\geq \tfrac1{n-1}$;
\item 
One-point crossover: $\tfrac4n +O(n^{-2})\geq \d(\nu)\geq \tfrac1{n+1}$;
\item 
Uniform crossover: $\tfrac4n +O(n^{-2})\geq \d(\nu)\geq \tfrac{1-2^{-n+1}}{n-1}$;
\item Bernoulli($q$) model: $\tfrac{4(1-\left(1-q/2\right)^n 
)}n +O(n^{-2})\geq \d(\nu)\geq\tfrac{1-(1-q)^{n}-q^n}{n-1}$.\end{enumerate}
\end{theorem}
Inequality \eqref{bord} can be seen as a nonlinear version of a logarithmic Sobolev inequality; see, e.g.,  \cite{DS} for background on logarithmic Sobolev inequalities in the usual Markov chain setting.  
Remarkable works have been devoted to the study of functional inequalities of the form \eqref{bord} in the Boltzmann equation literature; see, e.g., \cite{CC,Vil} and \cite{villani2002review,Cercignaniconj} for an overview. 
In our combinatorial setting, we shall establish the entropy production inequalities of Theorem \ref{main} by proving a new set of inequalities for the entropy functional of a product measure. More precisely, we shall first observe (see Lemma \ref{le2}),  that $\d(\pi,\nu)\geq \kappa(\pi,\nu)$, where 
$\kappa(\pi,\nu)$ denotes the largest possible constant $\kappa>0$ such that the inequality
\begin{equation}\label{olent5}
 \sum_{A} \nu(A)\,\left(\ent(f_A) + \ent(f_{A^c})\right)\leq (1-\kappa)\,\ent(f)
\end{equation}
holds for all $f\in\cS_\pi$.  Define  
\begin{equation}\label{bord3}
\k(\nu)=\inf_\pi \k(\pi,\nu)\,,
\end{equation}
where the infimum is taken over all product measures on $X_1\times\cdots\times X_n$ and over all choices of the spaces~$X_i$.  Then, since $\d(\nu)\ge\k(\nu)$, we can obtain the lower bounds
in Theorem~\ref{main} by computing~$\k$.  The following result, which may be of independent interest, does this for all of the above recombination models.
\begin{theorem}\label{main2}
For the recombination models defined above, the constant $\k(\nu)$ satisfies: 
\begin{enumerate}[1)]
\item
Single site recombination: $ \k(\nu)=\tfrac1{n-1}$;
\item 
One-point crossover: $\k(\nu)= \tfrac1{n+1}$;
\item 
Uniform crossover: $\k(\nu)= \tfrac{1-2^{-n+1}}{n-1}$;
\item Bernoulli($q$) model: $\k(\nu)=\tfrac{1-(1-q)^{n}-q^n}{n-1}$.\end{enumerate}
\end{theorem}
Inequality \eqref{olent5} expresses a generalized subadditivity property of the entropy functional for a product measure. As discussed in  Section \ref{proofs} below, it can be understood
as a refinement of Shearer's inequality for Shannon's entropy \cite{chung1986some}. We refer also to \cite{Friedgut,BB,MT} for other interesting extensions and applications of Shearer's inequality. Our setting is somewhat non-standard because of the restriction to $f\in\cS_\pi$. It is important to note that without it there cannot be a positive constant $\k$ in \eqref{olent5}. Indeed, 
by taking $f$ of the form $\prod_i f_i$ for some nontrivial single site marginals $f_i$, one has $\ent(f_A) + \ent(f_{A^c})=\ent(f)$ for any $A$ so that \eqref{olent5} can only hold with $\kappa=0$.



As mentioned earlier, the lower bounds in Theorem \ref{main} are a consequence of Theorem \ref{main2} and the fact that $\d(\nu)\geq \kappa(\nu)$. The upper bounds on the other hand will follow by exhibiting an explicit test function in \eqref{bord}. 

Finally, we turn to the consequences of Theorem \ref{main} for the convergence to equilibrium of both the continuous time and the discrete time evolutions. The following result shows that in all models considered above one has exponential decay in relative entropy with rate $\k(\nu)$  and that, if we insist on uniformity in the initial state $p$, this decay rate is optimal up to a constant factor.
 \begin{corollary}\label{corodio}
Consider the recombination model with distribution $\nu$, and let $\k(\nu)$ be as in Theorem \ref{main2}. For any initial state $p$, let $\pi$ denote the associated product measure \eqref{station}. Then, in continuous time one has 
\begin{equation}\label{dente1}
H(p_t\tc\pi)\leq e^{-\kappa(\nu)t}H(p\tc \pi)\,,\quad t\geq 0. 
\end{equation}
Similarly, in discrete time 
\begin{equation}\label{dente}
H(p^{(k)}\tc\pi)\leq (1-\kappa(\nu))^k\,H(p\tc \pi)\,, \quad k\in\bbN.
\end{equation}
Moreover, there exists an initial state $p$ 
such that, 
if $\g(\nu)$ denotes the upper bound on $\d(v)$ appearing in Theorem \ref{main}, then 
\begin{equation}\label{dente3}
\at{\dert H(p_{t}\tc\pi)}{t=0^+}
\geq -\g(\nu)H(p\tc \pi)\,,\quad H(p^{(1)}\tc\pi)\geq (1-\g(\nu))\,H(p\tc \pi)\,.
\end{equation}
These bounds are tight in the sense that there exists a constant $C>0$, independent of $n$, 
such that $ \g(\nu)\leq C\k(\nu)$ for all models 1)--3). In the case of the Bernoulli($q$) model, the bound $\g(\nu)\leq C\k(\nu)$ holds provided $q\geq n^{-2}$.  
\end{corollary}
%
The proofs of Theorem \ref{main}, Theorem \ref{main2} and Corollary \ref{corodio} are deferred to Section \ref{proofs}. In Section \ref{setup} we formulate the problem of 
entropy production estimates in the much more general setting of reversible quadratic systems.  In Section \ref{examples} we present some examples of reversible quadratic systems, including the recombination models defined above.

\section{Reversible quadratic systems}\label{setup}
Following \cite{Sinetal}, we introduce a general framework, which includes the recombination models as  
special cases. We shall work in continuous time, but as we will see a translation to the discrete time 
setting is immediate\footnote{\cite{Sinetal} worked in discrete time, and also restricted attention to {\it symmetric\/} quadratic systems, which are reversible quadratic systems with $\mu$ uniform.}.  
\subsection{Setup}
Let $\cX$ be a finite space, and call $\cP(\cX)$ the set of probability measures on $\cX$. We refer to, e.g., \cite{norris} for background on Markov chains. 
A {\it reversible quadratic system (RQS)\/} on $\cX$ is defined by a
pair $(\cG,\mu)$ where \begin{enumerate}
\item $\cG$ is the infinitesimal generator of a Markov chain with state space $\cX\times \cX$, i.e., 
$\cG(\si,\si';\t,\t')\geq 0$ for all $(\si,\si')\neq (\t,\t')$ and 
\begin{equation}\label{gene1}
\cG(\si,\si';\si,\si')= -\sum_{(\t,\t')\neq (\si,\si')} \cG(\si,\si';\t,\t').
\end{equation}
\item $\mu\in\cP(\cX)$ satisfies $\mu(\si)>0$ for all $\si\in\cX$ and is such that $\mu\otimes \mu$ is 
reversible for $\cG$, i.e., for all $\si,\si',\t,\t'\in\cX$ one has
\begin{equation}\label{reve1}
\mu(\si)\mu(\si')\cG(\si,\si';\t,\t')= \mu(\t)\mu(\t')\cG(\t,\t';\si,\si').
\end{equation}
Moreover, we assume that $\cG$ has the symmetry
\begin{equation}\label{symme1}
\cG(\si,\si';\t,\t')= \cG(\si',\si;\t',\t),
\end{equation}
for all $\si,\si',\t,\t'\in\cX$.
 \end{enumerate}
 The quantity $\cG(\si,\si';\t,\t')$ is interpreted as the rate at which the pair $(\t,\t')$ is produced from a collision (or ``mating") between 
 $\si$ and $\si'$.  Below, $\cG$ will be mostly of the form $\cG=Q-\ind$, where $Q$ is a  Markov kernel on $\cX\times\cX$ with reversible measure $\mu\otimes\mu$. It is important to note that the Markov chain on the  product space $\cX\times \cX$ defining the RQS is not assumed to be irreducible, and therefore $\mu$ can be taken to be any of the possibly many measures such that $\mu\otimes\mu$ is reversible for $\cG$. Indeed, in most cases to be considered below, $\cG$ will not be irreducible and we shall use that freedom in selecting $\mu$.  
 
 The dynamics of the system are specified by the equation
  \begin{equation}\label{sys1}
  \dert \,p_t(\t)
= \sum_{\t'\in\cX} \Phi[p_t](\t,\t')\,,
\end{equation}
where we define, for any $p\in\cP(\cX)$,
\begin{equation}\label{qp}
\Phi[p](\t,\t'):=\sum_{\si,\si'\in\cX}
p(\si)p(\si')\cG(\si,\si';\t,\t'). 
\end{equation}
We consider equation \eqref{sys1} with the initial condition $p_0 = p$ for some given $p\in\cP(\cX)$.
By \eqref{gene1}, for any fixed $\si,\si'$ one has $\sum_{(\t,\t')}\cG(\si,\si';\t,\t')=0$, which implies the conservation law 
\begin{equation}\label{cl1}
\sum_{\t\in\cX}p_t(\t)=1,
\end{equation} for all  $t\geq 0$. Thus equation \eqref{sys1} gives a well defined evolution of the state of the system. In general there are many other conservation laws in the system \eqref{sys1}; see part 3 of Proposition \ref{genstat}  below. 
Existence and uniqueness of the solution of \eqref{sys1} for any $p\in\cP(\cX)$ can be established in a standard way; see, e.g., \cite{Baakeetal}.  

From reversibility \eqref{reve1}, setting $f(\si)=p(\si)/\mu(\si)$ one finds
\begin{equation}\label{sys02}
\Phi[p](\t,\t') = \sum_{\si,\si'\in\cX}
\mu(\t)\mu(\t')\cG(\t,\t';\si,\si')\left[f(\si)f(\si') - f(\t)f(\t')
\right].
\end{equation}
Therefore, in terms of $f_t(\si)=p_t(\si)/\mu(\si)$, equation \eqref{sys1} becomes
  \begin{equation}\label{sys2}
  \dert \,f_t(\t)
= \sum_{\si,\si',\t'\in\cX}\mu(\t')\cG(\t,\t';\si,\si')\left[f_t(\si)f_t(\si') - f_t(\t)f_t(\t')
\right].
\end{equation}
\begin{remark}\label{remo1}
RQS include linear evolutions associated to Markov chains as a special case. For instance if $\cG_0$ is a Markov generator with state space $\cX$, with reversible measure $\mu$, then the expression 
$$
\cG(\si,\si';\t,\t')= \cG_0(\si,\t)\ind(\si'=\t')+ \cG_0(\si',\t')\ind(\si=\t)\,,
$$ 
defines a RQS $(\cG,\mu)$. In this case the evolution is linear, and coincides with the Markov chain generated by $\cG_0$, i.e., $p_t = p\,e^{t\cG_0}$.  
\end{remark}

\begin{remark}\label{remo2}
Given the generator $\cG$, let $S(\t,\t';\si,\si') = \tfrac12(\cG(\t,\t';\si,\si')+\cG(\t',\t;\si,\si'))$. Note that this does not define a Markov generator, since off-diagonal elements need not be nonnegative. However, it is not hard to check that 
\begin{gather*}\cG_{\rm sym}(\t,\t';\si,\si') = \begin{cases}
S(\t,\t';\si,\si')& \text{if }(\si,\si')\neq (\t',\t), 
(\t,\t')\\ 
\cG(\t,\t';\t',\t) &  \text{if }(\si,\si')= (\t',\t)\\
S(\t,\t';\t,\t') +S(\t,\t';\t',\t)-  \cG(\t,\t';\t',\t)&  \text{if }(\si,\si')= (\t,\t')
\end{cases}
\end{gather*}
{\em does} define a Markov generator, i.e.,  for all $\t,\t'\in\cX$ one has  $\sum_{\si,\si'}\cG_{\rm sym}(\t,\t';\si,\si')=0$, and 
$\cG_{\rm sym}(\t,\t';\si,\si')\geq 0$ if $(\si,\si')\neq (\t,\t')$.
From \eqref{sys02} and \eqref{symme1} it is not hard to see that $\Phi[p](\t,\t')$ is unchanged if $\cG$ is replaced by $\cG_{\rm sym}$. 
In particular, we may and will assume without loss of generality, that the generator $\cG$ 
satisfies the symmetry
\begin{equation}\label{genesym} 
\cG(\t,\t';\si,\si')=\cG(\t',\t;\si,\si'), \qquad \text{if } (\si,\si')\neq (\t',\t),(\t,\t').
\end{equation}
By reversibility \eqref{reve1}, from \eqref{genesym} one has the further symmetry \begin{equation}\label{genesym2} \cG(\si,\si';\t,\t')=\cG(\si,\si';\t',\t) , \qquad \text{if } (\si,\si')\neq (\t',\t),(\t,\t').
\end{equation}
In most examples below, $\cG$ has the form $\cG=Q-\ind$ for some Markov kernel $Q$. In this case 
the symmetry \eqref{genesym} is equivalent to $ Q(\si,\si';\t,\t')=Q(\si,\si';\t',\t)$, for all $(\si,\si')\neq (\t',\t),(\t,\t')$.
\end{remark}
\subsection{Entropy and stationary states}
Reversible quadratic systems satisfy an analogue of Boltzmann's H theorem from kinetic theory, which we summarize in Proposition \ref{genstat} below. Call $\r\in\cP(\cX)$ {\it stationary\/} if $\Phi[\r]=0$. Note that by \eqref{sys02} the reference measure $\mu$ is stationary.
Let $\cP_+(\cX)$ denote the set of $\r\in\cP(\cX)$ such that $\r(\si)>0$ for all $\si\in\cX$. 
For any RQS $(\cG,\mu)$, one has the following facts. 
\begin{proposition}\label{genstat} 
\begin{enumerate}[1)] 
\item For any initial state $p\in\cP(\cX)$,
 \begin{align}
 \label{entroprod1}
 \dert H(p_t\tc \mu)= -D(f_t,f_t),\end{align}
where $f_t(\si)=p_t(\si)/\mu(\si)$ and, for any $f,g:\cX\mapsto[0,\infty)$:
\begin{align*}
D(f,g):=\frac14  \sum_{\si,\si',\t,\t'\in\cX}
\mu(\t)\mu(\t')\cG(\t,\t';\si,\si')\left[f(\si)f(\si') - f(\t)f(\t')
\right]\log\frac{g(\si)g(\si')}{g(\t)g(\t')}.
\end{align*}
In particular, $D(f_t,f_t)\geq 0$ for all $t\geq 0$ and $D(f_t,f_t)=0$ iff $p_t$ is stationary.

\item  $\r\in\cP(\cX)$ is stationary iff $f(\si)=\r(\si)/\mu(\si)$ satisfies
$$
f(\si)f(\si') = f(\t)f(\t'),
$$
for all $\si,\si',\t,\t'\in\cX$ such that $\cG(\t,\t';\si,\si')>0$. In particular,  $\r\otimes \r$ is reversible for $\cG$ iff $\r$ is stationary.

\item  If $\r\in\cP_+(\cX)$ is stationary, then for any initial condition $p\in\cP(\cX)$, one has
 \begin{align}
 \label{invar}
 \dert \sum_{\si\in\cX}p_t(\si)\log (\r(\si)/\mu(\si))=0.
\end{align}
\end{enumerate}
\end{proposition}
\begin{proof}
Differentiating and using \eqref{cl1} one has
 \begin{align}
\dert H(p_t\tc \mu)
&= \sum_{\t\in\cX} \left[\dert p_t(\t)\right] \log \frac{p_t(\t)}{\mu(\t)}\nonumber \\
\;\;\;\;& = \sum_{\si,\si',\t,\t'\in\cX}
\mu(\t)\mu(\t')\cG(\t,\t';\si,\si')\left[f_t(\si)f_t(\si') - f_t(\t)f_t(\t')
\right]\log{f_t(\t)}. \label{entprod0}
\end{align}
By the symmetry \eqref{genesym}, one can replace $\log{f_t(\t)}$ above with $\log(f_t(\t)f_t(\t'))$ at the cost of a factor $1/2$. Finally, by reversibility one obtains \eqref{entroprod1}.
Clearly, $D(f,f)\geq 0$ and $D(f,f)=0$ iff $f$ satisfies
$f(\si)f(\si') = f(\t)f(\t')$, 
whenever $\cG(\t,\t';\si,\si')>0$. From this and \eqref{sys02} it follows that $D(f_t,f_t)=0$ iff $p_t$ is stationary. This proves part 1. 

Part 2 follows in the same way since if $p$ is stationary then $p_t=p$ for all $t$. 

Finally, for part 3, reasoning as above, if $g(\si)=\r(\si)/\mu(\si)$, $\r\in\cP_+(\cX)$, then 
$$
\dert \sum_{\t\in\cX}p_t(\t)\log (g(\t)) = -D(f_t,g).
$$
Since $\r$ is stationary, one has that $D(f_t,g)=0$ by part 2. 
\end{proof}

Let us remark that by Proposition \ref{genstat}, for any initial condition $p\in \cP(\cX)$, one has that 
$H(p_t\tc\mu)$ is monotone non-increasing, and therefore has a limit. This alone does not imply that $p_t$ converges. In fact, 
at this level of generality it may be hard to give a complete characterization of the limit points of $p_t$ as $t\to\infty$, as the initial state $p$ varies in $\cP(\cX)$. However, 
by compactness one has that along some subsequence $t_n\to\infty$, $p_{t_n}\to\r$ for some  $\r\in\cP(\cX)$. As already observed in \cite[Theorem 2]{Sinetal}, if one knows that $\r$ is stationary and has full support, i.e., $\r\in\cP_+(\cX)$, then it follows by Proposition \ref{genstat} that $p_t\to \r$, $t\to\infty$. To see this, observe that
$$
H(p_t\tc \r) = H(p_t\tc \mu)-p_t[\log(\r/\mu)].
$$ 
From Proposition \ref{genstat} part 3, $p_t[\log(\r/\mu)]$ is independent of $t$ and therefore equals $H(\r\tc\mu)$. On the other hand one must also have that   $H(p_t\tc\mu)$ decreases to $H(\r\tc\mu)$, and therefore $H(p_t\tc \r)\to 0$, which implies $p_t\to\r$ by \eqref{pinsker}. 

For all the examples of RQS to be discussed below we shall not need to appeal to the above abstract argument. In fact, for any initial distribution $p\in\cP(\cX)$ we shall always be able to identify an explicit limit point $\mu\in\cP_+(\cX)$ such that 
$H(p_t\tc\mu)\to0$, $t\to\infty$. This in turn implies the convergence $p_t\to\mu$ by \eqref{pinsker}. To quantify the convergence of relative entropy we proceed as follows. 

\begin{definition}\label{defent}
Given the  RQS $(\cG,\mu)$ on $\cX$, we define $\cF=\cF(\cG,\mu)$ as the set of  functions
$f:\cX\mapsto[0,\infty)$ such that $\mu[f]=1$,
and \
\begin{align}\label{conslaw}
\mu[f\log(\r/\mu)] 
= \mu[\log(\r/\mu)], 
\end{align}
for all stationary $\r\in\cP_+(\cX)$.
Moreover, we say that the RQS satisfies the {\em entropy production estimate} with constant $\d>0$ if for all 
$f\in\cF$ one has 
 \begin{align}\label{entrobo}
D(f,f)\geq \d\,\ent(f),
\end{align}
where $\ent(f)=\mu[f\log f]-\mu[f]\log\mu[f]$. 
\end{definition}

\begin{proposition}\label{entroprop}
Suppose the RQS $(\cG,\mu)$ satisfies the entropy production bound with constant $\d>0$. Then any initial state $p\in\cP(\cX)$ such that 
$p/\mu\in\cF$ satisfies 
\begin{align}\label{expodeco}
H(p_t\tc\mu)\leq e^{-\d \,t}H(p\tc \mu)\,,
\end{align}
for all $t\geq 0$. 
\end{proposition}
\begin{proof}
For any $f\geq 0$ such that $\mu[f]=1$ one has $H(f\mu\tc \mu)=\ent(f)$. Thus, by Proposition \ref{genstat} part 1, it suffices to show that \eqref{entrobo} holds for all functions $f_t=p_t/\mu$, $t\geq 0$. 
Since by assumption $p/\mu\in\cF$ one has $p[\log(\r/\mu)]= \mu[\log(\r/\mu)]$, and by Proposition \ref{genstat} part 3,
the same holds for $p_t$, for all $t\geq 0$. In particular $f_t\in\cF$, for all $t\geq 0$. 
\end{proof}

In practical applications our approach may be summarized as follows. 
Given the initial $p\in\cP(\cX)$, the goal will be to find a   stationary measure $\mu\in\cP_+(\cX)$ such that $p[\log(\r/\mu)]= \mu[\log(\r/\mu)]$, for all stationary $\r\in\cP_+(\cX)$ and then to prove the entropy production bound of Definition \ref{defent} for the RQS with that choice of $\mu$. Notice that in this framework the rate of decay $\d$ may depend on the initial value $p\in\cP(\cX)$ but only through the selected equilibrium point $\mu\in\cP_+(\cX)$.   

%
%

The functional inequality \eqref{entrobo} can be interpreted as a nonlinear version of the logarithmic Sobolev inequality for Markov chains. In our setup, one can actually write it in the form of a classical logarithmic Sobolev  
inequality in the product space $\cX\times\cX$, as follows. 
For any $f,g:\cX\mapsto[0,\infty)$, write $F(\si,\si')=f(\si)f(\si')$ and 
$G(\si,\si')=g(\si)g(\si')$. Then, as in the proof of Proposition \ref{genstat} one has
 \begin{align}\label{covFG} 
D(f,g) = -\frac12\,\widehat\mu\,[(\cG F)\log G],
\end{align}
where we use the notation $\widehat\mu=\mu\otimes\mu$, and $\cG F$ denotes the usual linear action of $\cG$ on $F$.
In particular, \eqref{entrobo} now becomes
 \begin{align}\label{entro}
-\widehat\mu\,[(\cG F) \log F] \geq \d\,\ent_{\widehat\mu}(F),
\end{align}
where $\ent_{\widehat\mu}(F) = \widehat\mu\,[F\log F] - \widehat\mu\,[F]\log (\widehat\mu\,[F]) = 2\ent(f)$. 
The inequality \eqref{entro} is often referred to in the Markov chain literature as a ``modified log-Sobolev inequality"; see, e.g., \cite{PPP,BT}. An important difference to keep in mind here with respect to the usual Markov chain setup is that we do not assume irreducibility, and therefore \eqref{entro} in general cannot hold for all $F$. Indeed, in our setting we require this to hold only for $F$ of the form $F(\si,\si')=f(\si)f(\si')$, where $f\in\cF$.    
\subsection{Linearized problem and spectral gap}
As in kinetic theory, see, e.g., \cite{villani2002review}, to gain insight into the functional inequality \eqref{entrobo} it is natural to investigate the linearized problem for near-to-equilibrium densities, i.e.,  $f=1+\e\,\phi$ for some $\phi:\cX\mapsto\bbR$ such that $\mu[\phi]=0$ with small $\e> 0$. 
\begin{lemma}\label{linear}
Let $f=1+\e\,\phi$ for some $\phi:\cX\mapsto\bbR$ such that $\mu[\phi]=0$. Then, as $\e\to 0$ one has 
 \begin{align}\label{linear1}
\lim_{\e\to 0}\e^{-2} \ent(f) =\frac12 \mu[\phi^2]\,,\quad \lim_{\e\to 0}\e^{-2} D(f,f) = - \mu[(\G\phi) \phi],
\end{align}
where $\G$ is the linear operator defined by 
$$
\G(\t,\si) = \sum_{\si',\t'}\mu(\t')[\cG(\t,\t';\si,\si')+\cG(\t,\t';\si',\si)]\,.
$$
\end{lemma}  
\begin{proof}
By expanding in the parameter $\e$ and retaining only terms up to order $\e^2$ it is not hard  to check that 
$\e^{-2} \ent(f) \to \frac12\mu[\phi^2]$, as $\e\to 0$. Similarly, neglecting terms of order $o(\e^2)$, $D(f,f)$ is given by
$$
-\,\e^2\sum_{\si,\si',\t,\t'}\mu(\t)\mu(\t')\cG(\t,\t';\si,\si')\phi(\t)[\phi(\si)+\phi(\si')].
$$
The latter expression equals $-\e^2\mu[(\G\phi) \phi]$.
\end{proof}
Thus, the linearized version of inequality \eqref{entrobo} states that 
\begin{align}\label{linearbo}
-\mu[(\G\phi ) \phi]\geq \frac\d{2} \mu[\phi^2]\,.
\end{align}
We note that 
if $\psi = \log(\r/\mu)$ for some stationary $\r\in\cP_+(\cX)$, then, as in the proof of Proposition \ref{genstat}, one 
finds $\G \psi = 0$.  This is not in contradiction with \eqref{linearbo}. Indeed, the condition $f=1+\e \phi\in\cF$ is equivalent to requiring that $\phi$ is orthogonal 
in $L^2(\mu)$ to the constant functions and to all conserved quantities $\psi=\log(\r/\mu)$ for stationary $\r\in\cP_+(\cX)$. 
Thus it is meaningful to study \eqref{linearbo} for all $\phi$ restricted to this class. This may be interpreted as a spectral gap bound, as follows. 

Suppose that $\cG=Q-\ind$, where $Q$ is a Markov kernel on $\cX\times \cX$. Then it is not hard to see that 
\begin{align}\label{gaq}
\G(\t,\si)=2 K(\t,\si) - \d(\si,\t) - \mu(\si),
\end{align}
where we introduce the Markov kernel 
\begin{align}\label{gaq1} K(\t,\si):=\tfrac12 \sum_{\si',\t'}\mu(\t')[Q(\t,\t';\si,\si')+Q(\t,\t';\si',\si)]\,.\end{align}
The kernel $K$ is reversible with respect to $\mu$. 
Moreover, if $\psi = \log(\r/\mu)$ for a stationary $\r\in\cP_+(\cX)$, then from $\G\psi=0$ and \eqref{gaq} we obtain the eigenvalue equation $$K\bar\psi = \tfrac12 \bar\psi,$$ 
where $\bar\psi:=\psi-\mu[\psi]$. Inequality \eqref{linearbo}
is then equivalent to 
\begin{align}\label{linearbo*}
\mu[(K\phi ) \phi]\leq \left(\tfrac12-\tfrac\d{4}\right) \mu[\phi^2]\,,
\end{align}
for all $\phi$ orthogonal 
in $L^2(\mu)$ to the constant functions and to the conserved quantities $\psi=\log(\r/\mu)$ as above.
From the obvious inequality $-\mu[(\G\phi )\phi]\geq 0$ one obtains for free that, apart from the trivial eigenvalue $1$, all eigenvalues of $K$ must be at most  $\tfrac12$. Thus, \eqref{linearbo*} says that, besides the eigenvalues associated to the conserved quantities, all other eigenvalues of $K$ are at most $\tfrac12-\tfrac\d{4}$. In some special cases one can fully diagonalize the operator $K$ and compute the optimal constant $\d$ in \eqref{linearbo*}.  It should be noted that, while the entropy production bound \eqref{entrobo} always implies the spectral gap \eqref{linearbo*} via Lemma \ref{linear}, the converse is of course not true.

\subsection{Continuous vs.\ discrete time}
If $\cG=Q-1$, where $Q$ is a Markov kernel on $\cX\times\cX$, then one can 
define the discrete time RQS as follows. Set 
\begin{align}\label{dtime}
\Psi[p](\t)=\sum_{\si,\si',\t'}p(\si)p(\si')Q(\si,\si',\t,\t').
\end{align}
This defines a map $\Psi:\cP(\cX)\mapsto\cP(\cX)$, whose $k$-th iterate $p^{(k)}=\Psi[p^{(k-1)}]$ describes the state of the system after $k$ steps, with initial state $p^{(0)}=p$. The entropy production estimate \eqref{entrobo} now takes the form
\begin{align}\label{expodecod}
H(p^{(k)}\tc\mu)\leq (1-\d)\,H(p^{(k-1)}\tc \mu)\,,
\end{align}
for all $k\in\bbN$. 
 Moreover, the same general remarks about stationary states and convergence to equilibrium apply in the discrete time setting; see also \cite{Sinetal}.
 
However, we caution the reader that, in contrast with the case of  linear evolution associated to a Markov chain, here there is a more pronounced difference 
between the discrete time evolution 
and the continuous time evolution. To briefly address this point, let us regard the operation $\Psi[p]$ as the product $p\odot p$, where we define, for all $p,q\in\cP(\cX)$ the new probability on $\cX\times\cX$  
$$
(p\odot q)(\t,\t')=\sum_{\si,\si'}p(\si)q(\si')Q(\si,\si',\t,\t').
$$
Using 
the symmetry \eqref{genesym} one sees that the product $p\odot q$ is commutative. However, it is not in general associative. This is the main source of difficulty in the explicit construction of the continuous time evolution $p_t$, in contrast with the simple iterations $p^{(k)}$ of the discrete time process.  
The form $p\odot q$ is the analogue of the Wild convolution product in the Boltzmann equation literature. A solution of the continuous time system can be constructed using suitable sums over so-called ``McKean trees", which encode the various ways of taking products, such as $p\odot(p\odot(p\odot p))$ or 
$(p\odot p)\odot(p\odot p)$ and so on. This construction builds on the pioneering work of Wild and McKean; see \cite{CCG} and references therein. 
Our results below will not make use of this method; they will be based only on functional inequalities of the form \eqref{entrobo} or \eqref{expodecod}, and will not make any significant distinction between discrete and continuous time evolution.  

%

\section{Main examples}\label{examples}
We now turn to concrete examples of RQS. 

\subsection{Binary uniform crossover}
We begin with the simplest possible example.
Let $\cX=\{0,1\}^n$ for some fixed integer $n$, and 
\begin{equation}\label{qbits}
Q(\si,\si;\t,\t')=\frac1{2^n}\sum_{A\subset[n]}\ind(\t=\si'_A\si_{A^c},\t'=\si_A\si'_{A^c}).
\end{equation}
In words, we move from  $(\si,\si')$ to $(\t,\t')$ under a uniform crossover, i.e., $(\t,\t')$ is obtained from $(\si,\si')$ by picking a uniformly random $A\subset[n]$ and swapping the $A$-components $\si_A$ and $\si'_A$, leaving the rest unchanged.  
Then $\cG=Q-\ind$ defines the generator of the RQS.
For the reference measure $\mu$ we may choose any product of Bernoulli probability measures. 
Indeed, 
\eqref{reve1} holds if $\mu=\otimes_{i=1}\mu_i$, with $\mu_i$ arbitrary Bernoulli distributions. Now, given any initial state $p\in\cP(\cX)$, equation \eqref{sys1}
can be written as 
 \begin{equation}\label{pct1}
\dert \,p_t=
 \frac1{2^n}\sum_{A}  \,(p_{t,A}\otimes p_{t,A^c} - p_t)\,,\quad t\geq 0,
\end{equation}
with the initial condition $p_0=p$.  This is the model described in \eqref{pct} in the special case where 
the single site spaces $X_i$ all coincide with $\{0,1\}$ and 
the distribution $\nu$ is uniform. The marginals $p_i$ of $p$ are preserved by the evolution, i.e., $p_{t,i}=p_i$ for all $i,t$, and thus the natural candidate for convergence of $p_t$ as $t\to\infty$ is the product measure $\pi=  \otimes_{i=1}p_i$; see also Lemma \ref{simple} below. Our results imply that the entropy production bound \eqref{entrobo} holds with $\d\geq 1/n$, independent of the initial state $p\in\cP(\cX)$. In particular, 
 \begin{equation}\label{re1}
H(p_t\tc\pi)\leq H(p\tc\pi)\,e^{-t/n}
\end{equation}
for any initial distribution $p\in\cP(\cX)$ and any $n$. As we shall see in the proof of Corollary \ref{corodio} the $1/n$ behavior of the constant cannot be improved if we require uniformity of the decay rate as a function of the initial state $p$.  In this example the linear operator $K$ from \eqref{gaq1} can be fully diagonalized, and one finds that \eqref{linearbo*} holds with $\d=1$.  In particular, this is an example of a system with uniformly positive spectral gap but with a vanishing rate (as $n\to\infty$) in the exponential decay of  relative entropy.

\subsection{General recombination model}
The binary uniform crossover model can be readily generalized to the case where 
the uniform choice of $A$ in \eqref{qbits} is replaced by a given distribution $\nu$ on subsets of $[n]$, and the state space is taken as $\cX=X_1\times\cdots\times X_n$ with arbitrary finite single site spaces $X_i$. 
Again $\cG = Q-\ind$, where now
\begin{equation*}
Q(\si,\si;\t,\t')=\sum_{A\subset[n]}\nu(A)\ind(\t=\si'_A\si_{A^c},\t'=\si_A\si'_{A^c}).
\end{equation*}
The system equation \eqref{sys1} then coincides with  \eqref{pct}. 
As above it is not hard to check that \eqref{reve1} is satisfied by any product measure $\mu=\otimes \mu_i$, with $\mu_i$ an arbitrary probability measure on $X_i$. To define the RQS $(\cG,\mu)$, we fix one such $\mu$, with $\mu\in\cP_+(\cX)$. 
Say that $\nu$ is {\it nondegenerate\/} if for any $i,j\in[n]$ there is a positive probability that the random set with distribution $\nu$ separates $i$ and $j$.  
\begin{lemma}\label{simple}
If $\nu$ is nondegenerate, then 
 $\r\in\cP(\cX)$ is stationary iff $\r$ has the product form $\r=\otimes_{i=1}^n\r_i$ for some probability measures $\r_i$ on $X_i$. 
\end{lemma}
\begin{proof}
From Proposition \ref{genstat} it follows that $ \r\in\cP(\cX)$ is stationary iff it satisfies
$$
\r(\si)\r(\si') = \r(\si'_A\si_{A^c})\r(\si_A\si'_{A^c})\,,
$$
for all $\si,\si'\in\cX$ and all $A\subset[n]$ such that $\nu(A)>0$. 
In particular, any product measure $\r=\otimes_{i=1}^n\r_i$ is stationary. To prove the converse, 
notice that summing over $\si'\in\cX$ in the above equation one has that 
\begin{equation}\label{marges}
\r(\si)=\r_A(\si_A)\r_{A^c}(\si_{A^c}),
\end{equation}
for all $\si\in\cX$ and all $A\subset[n]$ such that $\nu(A)>0$, where $\r_A$ denotes the marginal of $\r$ on $A$. 
We prove by induction that for any set $B\subset[n]$ one has $\r_B(\si_B)=\prod_{i\in B}\r_i(\si_i)$. 
Plainly, this is true for all sets $B\subset [n]$ with $|B|=1$. Suppose that it is true for all sets $B\subset [n]$ with $1\leq |B|\leq k$.  Take $B'\subset[n]$ with $|B'|=k+1$, and choose $i,j\in B'$. 
By the nondegeneracy assumption, there is a set $A$ such that $i\in A$, $j\in A^c$, and $\nu(A)>0$. Applying \eqref{marges} with this choice of $A$ and taking the marginal over $B'$, one finds $\r_{B'}(\si_{B'})=\r_{A\cap B'}(\si_{A\cap B'})\r_{A^c\cap B'}(\si_{A^c\cap B'})$. Since $1\leq |A^c\cap B'|\leq k$ and $1\leq |A\cap B'|\leq k$ we may apply the inductive assumption to conclude.
%
%
\end{proof}
Clearly, all models discussed in the introduction satisfy the nondegeneracy assumption. Since the marginals $p_i$ of $p$ are preserved by the evolution, the natural candidate for convergence of $p_t$ is the product measure $\pi =  \otimes_{i=1}p_i$. As highlighted in Corollary \ref{corodio}, our analysis will show that, for any initial state $p\in\cP(\cX)$, 
 \begin{equation}\label{reo1}
H(p_t\tc\pi)\leq H(p\tc\pi)\,e^{-\k(\nu) t}
\end{equation}
for all $t\geq 0$, where the constant $\k(\nu)$ is as specified in Theorem \ref{main2}.  

We remark that we may always pretend that $\pi$ has full support, and take $\pi$ itself as the reference measure $\mu$. Indeed, this is equivalent to the condition that $p_i$ has full support on $X_i$ for all $i\in[n]$, and if that is not the case then we may simply replace the $X_i$ with subspaces $\bar X_i$ such that $p_i$ has full support on $\bar X_i$ and work with the RQS $(\cG,\pi)$ within the space $\bar\cX= \bar X_1\times\cdots\times \bar X_n$ instead of $\cX$; this ensures that now $ \pi\in\cP_+(\bar\cX)$.  


In the remainder of this section we briefly consider some more general examples of RQS that go
beyond our recombination examples.  In particular, these generalized models will admit stationary 
measures that are not product measures.  While our quantitative results on entropy production
do not so far extend to these cases, they serve as examples of natural open questions in this area.

\subsection{Nonlinear stochastic Ising model}
As a canonical example of a generalization that admits
nontrivial correlations in the equilibrium state, we introduce a natural nonlinear dynamics on 
the Ising model on a finite graph $G=(V,E)$, $V=[n]$. The Ising model is the probability measure 
$\mu=\mu_{G,\b,{\bf h}}$ on $\cX=\{-1,1\}^V$, given by
\begin{equation}\label{is1}
\mu(\si)=\frac1{Z}\,e^{-\b H(\si)+\sum_{i\in V}h_i\si_i}\,,\qquad H(\si)=-\sum_{ij\in E}\si_i\si_j .
\end{equation}
Here ${\bf h}=\{h_i\}_{i\in V}\in[-\infty,\infty]^V$ are the so-called ``external fields", $\b\in\bbR$ is a parameter (the ``inverse temperature"), and $Z=Z_{G,\b,{\bf h}}$ is the normalizing factor, or ``partition function." Infinite values of external fields encode so-called ``boundary conditions": if $h_i=+\infty$ (respectively, $-\infty$) then one has a $+$ (respectively, \ $-$) boundary condition at site $i$. Clearly, $\mu\in\cP_+(\cX)$ iff all external fields are finite. Below, we fix one such $\mu$ as reference measure.  We set $\cG=Q-\ind$, with the Markov kernel $Q$ defined by $Q=\sum_A\nu(A) Q_A$, where $\nu$ is a given distribution over subsets and for any $A\subset [n]$, 
  \begin{gather}
Q_A(\si,\si';\t,\t') = \a_A(\si,\si')\ind(\t=\si'_A\si_{A^c},\t'=\si_A\si'_{A^c})+ (1-\a_A(\si,\si'))
\ind(\t=\si,\t'=\si'),\nonumber\\
\a_A(\si,\si'):=\frac{\mu(\si'_A\si_{A^c})\mu(\si_A\si'_{A^c})}{\mu(\si'_A\si_{A^c})\mu(\si_A\si'_{A^c})+\mu(\si)\mu(\si')}\label{genis}
\,.
\end{gather}
Notice that $\a_A(\si,\si')$ can be written as the conditional $\mu\otimes \mu$ probability of the pair $(\si'_A\si_{A^c},\si_A\si'_{A^c})$ given the occurrence of either $(\si'_A\si_{A^c},\si_A\si'_{A^c})$ or $(\si,\si')$. It satisfies
$$
\a_A(\si,\si') = \frac1{1+e^{\b\phi_A(\si,\si')}}\,,\qquad \phi_A(\si,\si') = \sum_{i\in A,j\in A^c}(\si_i-\si'_i)(\si_j-\si'_j)\ind(ij\in E)\,.
$$
In particular, $\a_A$ is independent of the external fields ${\bf h}$. It is easily checked that for each $A$, one has the reversibility
\begin{equation}
\label{revmu}
\mu(\t)\mu(\t')Q_A(\t,\t';\si,\si')= \mu(\si)\mu(\si')Q_A(\si,\si';\t,\t'),
\end{equation}
for all $\si,\si',\t,\t'\in\cX$. Therefore $(\cG,\mu)$ defines a RQS. Since the kernel $Q$ is independent of the external fields ${\bf h}$, the RQS here is determined by the distribution $\nu$, the parameter $\b$ and the graph $G$ (and not by ${\bf h}$). 
Moreover, any Ising measure of the form \eqref{is1} satisfies \eqref{revmu} and therefore it is stationary for the RQS. 
A particularly interesting choice is the single site update $\nu(A)=\tfrac1n\ind(|A|=1)$, which can be interpreted as a nonlinear version of the usual Ising Gibbs sampler, or Glauber dynamics; see, e.g., \cite{LPW} for an introduction.  
One can prove the following characterization of the stationary distributions. 
\begin{lemma}\label{stationa}
Fix a graph $G$ with $n$ vertices, and $\b\in\bbR$. Assume $\nu(A)=\frac1n\ind(|A|=1)$.  Let $\mu\in\cP_+(\cX)$ be  as in \eqref{is1} with arbitrary external fields. A distribution $\r\in\cP(\cX)$ is stationary for the RQS $(\cG,\mu)$ if and only if $\r$ is of the form \eqref{is1} for some choice of ${\bf h}$.
\end{lemma}
\begin{proof}
We have seen that any $\r$ of the form \eqref{is1} satisfies \eqref{revmu} and it is therefore stationary. 
To prove the converse, observe that by Proposition \ref{genstat} part 2, one has that any stationary $\r$ must satisfy 
 \begin{equation}
\label{revmu1}
\frac{\r(\si)\r(\si')}{\mu(\si)\mu(\si')}= \frac{\r(\si'_{\{i\}}\si_{[n]\setminus\{i\}})\r(\si_{\{i\}}\si'_{[n]\setminus\{i\}})}{\mu(\si'_{\{i\}}\si_{[n]\setminus\{i\}})\mu(\si_{\{i\}}\si'_{[n]\setminus\{i\}})},
\end{equation}
for all sites $i\in[n]$, and for all $\si,\si'\in\cX$. 
Suppose first that $\r\in\cP_+(\cX)$.  We will show that,
if $f(\si)=\r(\si)/\mu(\si)$, then there exist
$\varphi_1,\dots,\varphi_n$ with $\varphi_j:\{-1,+1\}\mapsto(0,\infty)$ such that 
  \begin{equation}
\label{revmu3}
f(\si) = f(+)\prod_{j=1}^n\varphi_j(\si_j)\,,
\end{equation}
where $+$ denotes the configuration with all spins equal to $ +1$. To prove \eqref{revmu3}, for any $\si\in\cX$, $j\in[n]$, let $\si[j]\in\cX$ denote the configuration equal to $+1$ for all sites in $[j]=\{1,\dots,j\}$ and equal to $\si$ on $\{j+1,\dots,n\}$. Since $\r,\mu\in\cP_+(\cX)$, if $\si[0]=\si$, one has 
\[
f(\si)=f(+)\prod_{j=1}^n \frac{f(\si[j-1])}{f(\si[j])}.
\]
Taking $i=j$, $\si = \si[j-1]$ and $\si'=+$ in \eqref{revmu1}, and setting $ \varphi_j(\si_j) =\frac{ f(+^{j,\si_j})}{f(+)}$, where $+^{j,\si_j}$ denotes the configuration equal to $\si_j$ at $j$ and equal to $+1$ elsewhere, one has 
\[
\frac{f(\si[j-1])}{f(\si[j])} = \varphi_j(\si_j) \,,\quad j=1,\dots,n.
\]
This proves \eqref{revmu3}. Since $\varphi_j$ is a positive function of a single spin, it  can be written as $\varphi_j(\si_j)=e^{h_j\si_j + c_j}$ for some real numbers $h_j,c_j$. Therefore, from \eqref{revmu3} we obtain 
$\r(\si)={\rm const}\times \mu(\si)e^{\sum_{j\in V}\si_j h_j}$. 
This ends the proof for $\r\in\cP_+(\cX)$. 

If $\r$ does not have full support, then there must exist $A\subset[n]$ and a vector $v_A=(v_i)_{i\in A}\in\{-1,+1\}^{A}$ such that for all $\t\in\cX$, 
 \begin{equation}
\label{revmu5}
\r(\t)=\ind(\t_A=v_A)\r'(\t_{A^c}),
\end{equation}
for some probability $\r'$ on $\{-1,+1\}^{A^c}$ with full support.  (Here~$A$ is the set of vertices
with boundary conditions in~$\rho$.)
This can be seen as follows. It is not hard to check that if $\r\in\cP(\cX)\setminus\cP_+(\cX)$, then there exists  $\si\in\cX$, $i\in[n]$ such that $\r(\si)>0$ and $\r(\si^i)=0$. Then, taking $\si'$ equal to $-\si_i$ at $i$ and arbitrary otherwise, from \eqref{revmu1} one sees that $\r(\t)=0$ for all $\t\in\cX$ such that $\t_i=-\si_i$. Now, if $\r(\t')>0$ for all $\t'\in\cX$ such that $\t'_i=\si_i$, then \eqref{revmu5}
holds with $A=\{i\}$ and $v_i=\si_i$. Otherwise, restricting attention to configurations with the $i$-th spin equal to $\si_i$, one can repeat the above reasoning, and the desired conclusion follows by recursion.   

Once \eqref{revmu5} is established, one can repeat the argument given in the case of $\r\in\cP_+(\cX)$, by restricting to the subspace of $\si\in\cX$ such that $\si_A=v_A$. On this set the measure $\r$ has full support and  one obtains  again \eqref{revmu3}, this time with $f(+)$ replaced by $f(\xi)$ where $\xi$ equals $v_A$ on $A$ and $+$ on $A^c$, with $\si[j]$ interpolating from $\si$ to $\xi$, and with the product restricted  to $j\in A^c$.  
It follows that $\r$ has the form  \eqref{is1}, with $h_i\in\bbR$ for $i\in A^c$, $h_i=+\infty$ for $i\in A$ such that $v_i=+1$, and $h_i=-\infty$ for $i\in A$ such that $v_i=-1$.
\end{proof}

We remark that Lemma 3.2 assumes that $\mu\in\cP_+(\cX)$, i.e.\ that the corresponding external fields are all finite.
However, in order to analyze the Ising model with boundary conditions one may wish to take some of the external fields to have positive or negative infinite values. In this case one can define the RQS $(\cG,\mu)$ as above with the difference that now the state space $\cX$ is replaced by the restricted space of $\si\in\cX$ that are aligned with the boundary conditions at those sites which have infinite external field, so that $\mu$ has full support when restricted to this set. 
Then it is not hard to check that  Lemma 3.2 continues to hold, with the same proof.

Since marginals are conserved by the evolution, for any $p\in\cP(\cX)$ we expect that $p_t\to \mu$, where $\mu=\mu(p)$ is the equilibrium measure \eqref{is1} with external fields ${\bf h}$ such that the marginals $\mu_i$ coincide with the marginals $p_i$. It is standard to check that such a choice always exists. Equivalently, the measure $\mu=\mu(p)$ is characterized by the condition $$\mu[\log(\r/\mu)]=p[\log(\r/\mu)],$$ for all stationary $\r\in\cP_+(\cX)$, since by Lemma \ref{stationa} one has that $\log(\r/\mu)=\sum_{i}a_i \si_i$ for arbitrary coefficients $a_i$. Recall that this is the condition appearing in Definition \ref{defent}. 

In analogy with results on entropy decay for the Glauber dynamics (see, e.g., \cite{CapMenzTet} and references therein), we propose the following conjecture concerning convergence to equilibrium.
\begin{conjecture}\label{isingconj}
Assume $\nu(A)=\frac1n\ind(|A|=1)$.
There exists a constant $c>0$ such that for any graph $G$ with $n$ vertices and maximum degree $\D$, for any  $\b\in\bbR$ with $|\b|\leq c/\D$, and for any choice of external fields ${\bf h}$, the RQS $(\cG,\mu)$ with $\mu= \mu_{G,\b,{\bf h}}$ satisfies the entropy production estimate of Definition \ref{defent} with constant $\d=c/n$. 
\end{conjecture}

Note that if $\b=0$, then $\a_A\equiv \tfrac12$,  then we are back to the single site recombination model, for which Conjecture \ref{isingconj} holds by Corollary \ref{corodio}.
We turn now to the description of some possible variants of the nonlinear stochastic Ising model. 
\subsubsection{Foldings}
A further stochastic Ising model is obtained as follows. Given a pair $(\si,\si')\in\cX\times \cX$, let $B=B(\si,\si')$ denote the set of vertices where they agree: $B=\{i\in[n]:\,\si_i=\si_i'\}$. Then define the kernel
\begin{equation}
\label{folding}
Q(\si,\si';\t,\t')= \frac1{\cZ(\si,\si')}\,\mu(\si_B\t_{B^c})\mu(\si_B\t'_{B^c})\ind(\t'_{B^c}=-\t_{B^c})\,,
\end{equation}
where $\cZ(\si,\si')$ is the normalizing constant
$$
\cZ(\si,\si')=\sum_{\t\in\{-1,+1\}^{B^c}}\mu(\si_B\t_{B^c})\mu(\si_B(-\t_{B^c})).
$$
It is not hard to check that $\cG=Q-\ind$ defines a RQS with the required properties for any measure $\mu= \mu_{G,\b,{\bf h}}\in\cP_+(\cX)$ as in \eqref{is1}. Again $Q$ does not depend on the external field ${\bf h}$ and all choices of ${\bf h}$ produce a valid stationary state. In the case $\b=0$ one has that \eqref{folding} coincides with the binary uniform crossover example \eqref{qbits} above. 
We expect that for small $\b$ an estimate as in Conjecture \ref{isingconj} should hold for this model as well. 
The kernel $Q$ in \eqref{folding} is an example of a ``folding" transformation in the terminology introduced in \cite{BG}.  

\subsubsection{Adding a dissipative term }\label{dissi}
The previous models are conservative in the sense that single site marginals are constant in time. 
One can obtain a non-conservative evolution by adding a dissipative term as follows. The added terms can be interpreted as mutation operators in the context of genetic algorithms; see, e.g., \cite{goldberg}.
 
Fix a graph $G$ with vertex set $V=[n]$, $\b\in\bbR$ and a set of external fields ${\bf h}$, and let $\mu= \mu_{G,\b,{\bf h}}$ be the associated Ising Gibbs measure. Suppose that $W(\si;\t)$ is the usual Glauber 
dynamics kernel for $\mu$, i.e., 
\begin{equation}
\label{glaub}
W(\si;\t)= \frac1n\sum_{i=1}^n \left[\mu\left(\si^i\tc \{\si\}\cup\{\si^i\})\ind(\t=\si^i\right) + \mu\left(\si\tc \{\si\}\cup\{\si^i\}\right)\ind(\t=\si)\right]\,,
\end{equation}
where again, for any $\si\in\cX$, $\si^i\in\cX$ is obtained from $\si$ by reversing the spin at $i$.
Define
$$
\wt\cG(\si,\si';\t,\t')=(W(\si;\t)-\ind(\si=\t))\ind(\si'=\t')+ (W(\si';\t')-\ind(\si'=\t'))\ind(\si=\t)\,.
$$ 
As in Remark \ref{remo1}, the RQS $(\wt \cG,\mu)$ defines a linear evolution, namely the Glauber dynamics. 
If $\cG$ is the generator of the RQS $(\cG,\mu)$ introduced in \eqref{genis}, then  one can define a new nonlinear RQS $(\cG',\mu)$ with generator $\cG'=\cG+\wt\cG$. It is not hard to check that $(\cG',\mu)$ satisfies the required properties. Moreover, in this case one has the following quantitative convergence results. 

\begin{theorem}\label{isingdis}
Fix the graph $G$ with $n$ vertices, $\b\in\bbR$,  and a set of external fields ${\bf h}$. Let $p_t$ denote the evolution according to the RQS $(\cG',\mu)$ defined above with $\mu= \mu_{G,\b,{\bf h}}$.
\begin{enumerate}[1)]
\item There exists $c(\b,G)>0$ independent of ${\bf h}$ such that 
for all $p\in\cP(\cX)$, and all $t\geq 0$:
\begin{equation}
\label{dis1}
H(p_t\tc\mu)\leq e^{-c(\b,G)\,t}H(p\tc\mu)\,.
\end{equation}
\item There exists a constant $c>0$ independent of $G$, $\b$ and ${\bf h}$ such that
for all $\b\in\bbR$ with $|\b|\leq c/\D$, where $\D$ is the maximal degree of $G$, and for all  $p\in\cP(\cX)$, and all $t\geq 0$:
\begin{equation}
\label{dis2}
H(p_t\tc\mu)\leq e^{-c\,t/n}H(p\tc\mu)\,.
\end{equation}
\end{enumerate}
\end{theorem}
The main difference between the two estimates above is that the first is valid for any $\b$ and involves a constant $c(\b,G)$ that may be exponentially small as a function of the graph $G$ (see the proof below for an explicit expression), while the second is a bound with decay rate of order $1/n$ that is valid only at sufficiently high temperature.  In any case, Theorem \ref{isingdis} shows that in contrast with the conservative case, all initial states $p\in\cP(\cX)$ converge to the same 
equilibrium point $\mu$. Indeed, in this case the kernel $W$, which depends on the external fields, drives the system towards the equilibrium $\mu$. 
\begin{proof}[Proof of Theorem \ref{isingdis}]
For the first estimate, it is sufficient to prove that \eqref{entrobo} holds for all $f:\cX\mapsto[0,\infty)$ with $\d=c(\b,G)>0$. Notice that we do not restrict to any particular class of functions here. 
With the notation of \eqref{covFG} we have
$$
D(f,f) = -\frac12\widehat\mu\left[(\cG F+\wt\cG F)\log F\right] \geq -\frac12\widehat\mu\left[(\wt\cG F)\log F\right], 
$$
where we use the fact that $-\widehat\mu\left[(\cG F)\log F\right]\geq 0$, which follows from the representation in Proposition \ref{genstat}, part 1. Since the operator $\wt\cG$ corresponds to two independent Glauber dynamics on $\cX\times\cX$, from Proposition 1.1 and Lemma 2.2 in \cite{CapMenzTet} one has the modified logarithmic Sobolev inequality
$$
-\widehat\mu\left[(\wt\cG F)\log F\right]\geq c(\b,G)\,\widehat\mu\left[F\log F\right]\,,
$$
where $c(\b,G)=n^{-1}e^{-6|\b|\,|E|}$ and $|E|$ is the number of edges in $G$.
This proves part 1 since $\widehat\mu\left[F\log F\right]=2\ent(f)$.

To prove part 2 we use the same argument, but now observe that for some absolute constant $c>0$,  for  
$|\b|\leq c/\D$,  by Corollary 2.3 in \cite{CapMenzTet} one has
$$
-\widehat\mu\left[(\wt\cG F)\log F\right]\geq \frac{c}{n}\,\widehat\mu\left[F\log F\right]\,.
$$
\end{proof}
  
\subsection{Further examples}
Many more examples of RQS can be constructed by using suitable Markov chain generators in the product space $\cX\times \cX$. This allows one to construct natural and possibly useful nonlinear versions of various familiar stochastic processes such as random walks or card shuffling. We refer to \cite{Sinetal} for an interesting application to matchings in graphs. 

Our discussion in this paper is limited to the quadratic case where two independent samples from a given population interact to produce a new population. However, it is not difficult to generalize the setting by considering more than just two samples from the starting population. One may then obtain higher order nonlinear equations. For instance, a cubic version of the recombination process \eqref{pct} would take the form
 \begin{equation}\label{hpct}
\dert \,p_t=
 \sum_{A,B,C} \nu(A,B,C) \,(p_{t,A}\otimes p_{t,B}\otimes p_{t,C} - p_t)\,,\quad t\geq 0,
\end{equation}
where $A,B,C$ form a partition of $[n]$, and $\nu$ is a probability over such partitions.
Generalized recombination models of this kind have been recently considered in \cite{Baakeetal}. 
Further generalizations (which are not necessarily even mass-preserving) can be found in
the field of ``mass action kinetics" introduced in \cite{Feinberg,HornJackson}, which remains a very active area today.

\section{Entropy production estimate for recombinations}\label{proofs}
In this section we prove Theorem \ref{main}, Theorem \ref{main2} and Corollary \ref{corodio}. Most of the work goes into proving the lower bounds on the constant $\d(\nu)$ in Theorem \ref{main}. 
The first step consists in reducing this problem to the more tractable problem of controlling the constant $\k(\nu)$ in Theorem \ref{main2}.

Consider the RQS $(\cG,\mu)$ defined by $\cG=Q-\ind$ where  
\begin{equation}\label{qbitsgen}
Q(\si,\si;\t,\t')=\sum_{A}\nu(A)\ind(\t=\si'_A\si_{A^c},\t'=\si_A\si'_{A^c}),
\end{equation}
for some distribution $\nu$ on subsets of $[n]$, and   
$\mu=\otimes_{i=1}^n\mu_i$  an arbitrary product measure on the product space $\cX=X_1\times\cdots\times X_n$, where $X_i$ are given finite spaces. Assume that $\mu(\si)>0$ for all $\si\in\cX$. 

For any $A\subset[n]$, let $\cX_A=\prod_{i\in A}X_i$, and for any $f:\cX\mapsto [0,\infty)$, write $f_A$ for 
the function $f_A:\cX_A\mapsto[0,\infty)$ defined by
 \begin{equation}\label{marg}
f_A(\si_A)= \sum_{\si'\in\cX} \mu(\si') f(\si_A\si'_{A^c}).
\end{equation}
Notice that, if $f$ is a density with respect to $\mu$ (i.e., $\mu[f]=1$) then $f_A$ is the density of the marginal of $f\mu$ on $\cX_A$. It will often be convenient to regard $f_A$ as a function on the whole space $\cX$ simply by setting $f_A(\si)=f_A(\si_A)$. If $p\in\cP(\cX)$, we write $p_A$ for the marginal on $\cX_A$. Thus, if $f=p/\mu$, then $f_Af_{A^c}$ denotes the density of $p_A\otimes p_{A^c}$ with respect to $\mu$. When $A=\emptyset$, we set $f_{\emptyset}=1$. 
When $A=\{i\}$ is a singleton, we simply write $f_i=f_{\{i\}}$ for the single site marginal. 
\begin{lemma}
\label{le1}
Fix an arbitrary distribution $\nu$ on subsets of $[n]$. Let $\cS_\mu$ denote the set of $f:\cX\mapsto [0,\infty)$ such that $\mu[f]=1$ and $f_i = 1$ for all $i=1,\dots,n$. Suppose that there exists $\d>0$ such that 
\begin{equation}\label{lent1}
 \sum_{A} \nu(A)\,\mu[f_A f_{A^c}\log f]\leq (1-\d)\,\ent(f)
\end{equation}
for all $f\in\cS_\mu$.  
Then the 
RQS $(\cG,\mu)$ satisfies the entropy production estimate with constant $\d$, as defined in Definition \ref{defent}.
\end{lemma}
\begin{proof}
Rewrite the functional $D(f,f)$ as 
\begin{align*}
D(f,f)=-\sum_{A}\nu(A)\sum_{\t,\t'\in\cX}
\mu(\t)\mu(\t')\left(f(\t'_A\t_{A^c})f(\t_A\t'_{A^c}) - f(\t)f(\t')
\right)\log f(\t).
\end{align*}
With the notation $f_A$ for the marginal densities, and using the product structure of $\mu$, the above expression becomes 
\begin{equation}\label{lent2}
D(f,f)= \ent(f) - \sum_{A} \nu(A)\,\mu[f_A f_{A^c}\log f]\,.
\end{equation}
To conclude the proof it remains to show that $\cF\subset \cS_\mu$, where $\cF$ is the set of functions from  Definition \ref{defent}. Indeed, suppose that $f\in\cF$. Then for any stationary $\r\in\cP_+(\cX)$ one has $\mu[f\log(\r/\mu)]=\mu[\log(\r/\mu)]$. Let us show that $f_i=1$ for all $i\in[n]$. Take $\r\in\cP_+(\cX)$ of the form $\r= \otimes_{i=1}^n\r_i$ for some probability measures $\r_i$ on $X_i$. Then $\r$ is stationary. Choosing $\r_j=\mu_j$ for all $j\neq i$, one has  
$$
\mu[f_i\log(\r_i/\mu_i)]=\mu[\log(\r_i/\mu_i)]\,.
$$
Since the $\r_i$'s are arbitrary, it follows that $f_i=1$ for all $i$. 
\end{proof}
We remark that since the distribution $\nu$ in Lemma \ref{le1} is arbitrary the inclusion $\cF\subset\cS_\mu$ used in the previous proof may be strict. However, by Lemma \ref{simple}, if $\nu$ is  nondegenerate then one has that $\cF=\cS_\mu$ since the only stationary measures are of product form. 

Notice that 
\eqref{lent1} is equivalent to 
\begin{equation}\label{lent4}
 \sum_{A} \nu(A)\,\mu\left[(f_A f_{A^c}-f)\log {\frac{f_A f_{A^c}}f}\right]\geq \d\,\ent(f)\,.
\end{equation}
This follows from the fact that for any $A$ one has $$\mu\left[(f_A f_{A^c}-f)\log (f_A f_{A^c})\right]=0.$$ 
Thus, the largest constant $\d$ such that \eqref{lent1} holds for all $f\in\cS_\mu$ is precisely
the constant $\d(\pi,\nu)$ appearing in inequality~\eqref{bord} (with $\pi$ replaced by $\mu$).

The following observation 
allows us to reduce \eqref{lent4} to a more tractable expression. 
Suppose $f=p/\mu$ for some $p\in\cP(\cX)$.  Then 
$$
\mu[f_A f_{A^c}\log f] = - (p_A\otimes p_{A^c})\left[\log\left(\tfrac{p_A\otimes p_{A^c}}p\right)\right] + (p_A\otimes p_{A^c})\left[\log\left(\tfrac{p_A\otimes p_{A^c}}\mu\right)\right].
$$
Rearranging and using $\mu=\mu_A\otimes\mu_{A^c}$ one has the identity 
\begin{equation}\label{lent3}
\mu[f_A f_{A^c}\log f]= -H(p_A\otimes p_{A^c}\tc p) + H(p_A\tc \mu_A) + H(p_{A^c}\tc \mu_{A^c})\,.
\end{equation}
In particular, 
\begin{align*}\label{lent3}
\mu[f_A f_{A^c}\log f] & \leq  H(p_A\tc \mu_A) + H(p_{A^c}\tc \mu_{A^c}) \\
&= \mu[f_A \log f_A] + \mu[f_{A^c} \log f_{A^c}] = \ent(f_A) + \ent(f_{A^c})\,.
\end{align*}
Since the above holds for arbitrary product measures~$\mu$,
with the notation of Lemma \ref{le1} we have obtained the following criterion. 
\begin{lemma}
\label{le2}
Suppose that, for all $f\in\cS_\mu$,  
\begin{equation}\label{lent5}
 \sum_{A} \nu(A)\,\left( \ent(f_A) + \ent(f_{A^c})\right)\leq (1-\kappa)\,\ent(f)\,,
\end{equation}
with some $\kappa>0$. Then \eqref{lent1} 
holds with constant $\d=\kappa$.  In particular, if $\d(\nu)$ and $\k(\nu)$ are the constants introduced in Theorem \ref{main} and Theorem \ref{main2}, then 
for any distribution $\nu$,
\begin{equation}\label{pr1}
\d(\nu)\geq \k(\nu)\,.
 \end{equation}
\end{lemma}
We turn now to the analysis of the constant $\k(\nu)$. 
We start by recalling some preliminary inequalities.

\subsection{Tensorization and Shearer-type inequalities}
%
We recall that for any probability measure $\mu$, for any $A\subset [n]$, one has the decomposition
 \begin{align}\label{decompo}
  \ent(f)  =  \ent(\mu[f\tc A])  + \mu[\ent(f\tc A)],
\end{align}
where $\ent(f\tc A) = \mu\left[f\log (f/\mu[f\tc A])\tc A\right]$ denotes the entropy of $f$ with respect to the conditional probability measure $\mu[\cdot\tc A]$, obtained by conditioning $\mu$ on a given realization of  the variables $\si_{A}\in\cX_A$. The decomposition \eqref{decompo} is obtained by adding and subtracting $\mu[f\log\mu[f\tc A]]$ from $\ent(f)$. 
We note that if $\mu$ is a product measure, then $\mu[f\tc A]$ coincides with $f_A$ defined in~eqref{marg}. 

Let $\cA$ be an arbitrary family of subsets $A\subset[n]$.
Let $\deg_k(\cA)$ denote the degree of a site $k$ in $\cA$, i.e., the number of subsets $A\in\cA$ such that $A\ni k$, and set 
$$
n_-(\cA) = \min\{\deg_k(\cA)\,,\;k\in[n]\}\,,\quad n_+(\cA) = \max\{\deg_k(\cA)\,,\;k\in[n]\},
$$ 
for the minimal and maximal degrees, respectively. 
\begin{proposition}\label{shearer}
Let $\mu$ be  a product measure. For any family of subsets $\cA$ and any function $f\geq0$, 
\begin{align}\label{lent9}
n_-(\cA)\, \ent(f) \leq 
\sum_{A\in\cA} \mu[\ent(f\tc A^c)]\,.
 \end{align}
Equivalently, for any $\cA$ and any function $f\geq0$, 
 \begin{align}\label{lent10}
  \sum_{A\in\cA} \ent(f_A)\leq n_+(\cA)\, \ent(f).
\end{align}
\end{proposition}
\begin{proof}
The equivalence of \eqref{lent9} and \eqref{lent10} follows from \eqref{decompo} by passing from $\cA$ to the complementary set of subsets $\bar\cA = \{A^c, A\in\cA\}$. 
We prove \eqref{lent10}  as a consequence of the classical Shearer estimate for Shannon entropy \cite{chung1986some}. Suppose first that $\cA$ is  a {\it regular cover} of $[n]$, i.e., the union
of $A\in\cA$ is $[n]$ and the degrees $n(\cA):=\deg_k(\cA)$ are independent of $k\in[n]$. 
By homogeneity, we may assume 
$\mu[f]=1$. Call $Z=(Z_1,\dots,Z_n)$ the random variable with 
probability distribution $f\mu$, so that $f_A\mu_{A}$ is the law of the marginal $Z_A=(Z_i,\;i\in A)$; see \eqref{marg}.
For any $A$, the Shannon entropy $H(Z_A)$ of $Z_A$ satisfies
\begin{align}\label{shannon}
H(Z_A)&=-\sum_{\si_A}f_A(\si_A)\mu_{A}(\si_A) \log(f_A(\si_A)\mu_{A}(\si_A)) \nonumber \\ & = -\ent(f_A) -
 \sum_{\si_A}\sum_{i\in A}  f_A(\si_A)\mu_{A}(\si_A) \log (\mu_{i}(\si_i))\nonumber \\
 & = -\ent(f_A) -
 \sum_{i\in A} \sum_{\si_i} f_i(\si_i)\mu_{i}(\si_i) \log (\mu_{i}(\si_i))\nonumber \\
&=  - \ent(f_A) + \sum_{i\in A} H(Z_i) + \sum_{i\in A}\mu[f_i\log f_i].
\end{align}
Equivalently,
\begin{align}\label{shannon1}
\sum_{i\in A} H(Z_i) -H(Z_A)
= \ent(f_A) - \sum_{i\in A}\ent(f_i).
\end{align}
Shearer's estimate for the Shannon entropy 
states that
\begin{align}\label{shear1}
n(\cA)\, H(Z)\leq \sum_{A\in\cA} H(Z_A);
\end{align}
see, e.g., \cite{chung1986some} or \cite{BB} for a proof.
Therefore, summing over $A\in\cA$ in \eqref{shannon1} and using \eqref{shear1},  
\begin{align}\label{shea3}
  &\sum_{A\in\cA} \ent(f_A) - n(\cA)\sum_{i\in[n]}\ent(f_i)
  \leq  n(\cA)\sum_{i\in[n]}H(Z_i) - n(\cA) H(Z).
\end{align}
Applying \eqref{shannon1} with $A=[n]$ to the right hand side of~\eqref{shea3}, one obtains 
\eqref{lent10}. 

Suppose now that $\cA$ is arbitrary, so that it need not cover $[n]$ or have uniform degrees. 
Then one can add singleton sets to $\cA$ until one obtains a regular cover $\cA'$ such
that $n_+(\cA)=n(\cA')$.  It then follows that 
$$
\sum_{A\in\cA} \ent(f_A)\leq \sum_{A\in\cA'} \ent(f_A)\leq 
n(\cA')\, \ent(f) = n_+(\cA) \,\ent(f). 
$$
This ends the proof of \eqref{lent10}.
\end{proof}

By applying Proposition \ref{shearer} to the dyadic cover $\cA=\{A,A^c\}$ one obtains that, for any $A\subset [n]$ and any nonnegative function $f$:
\begin{align}\label{lent8}
\ent(f)&\leq \mu\left[\ent(f\tc A)\right] + \mu\left[\ent(f\tc A^c)\right];
\\
\ent(f_A) + \ent(f_{A^c}) &\leq \ent(f). \label{lent88}
\end{align}
The bounds \eqref{lent8} and \eqref{lent88} express respectively the well known tensorization and subadditivity properties of a product measure. 
%
%
%


 \subsection{Proof of Theorem \ref{main2}}
Let us point out first that a simple application of Proposition \ref{shearer} is not sufficient to prove Theorem \ref{main2}. For instance, consider the uniform crossover model where $\nu(A)=2^{-n}$ for any $A\subset[n]$. Then the left hand side of \eqref{lent5} becomes
 \begin{equation}\label{lent140}
2^{-n+1}\sum_A \ent(f_A)\,.
\end{equation}
If $f\in\cS_\mu$, then $$\ent(f_i)=\mu[f_i\log f_i]=0,$$ and, from Proposition \ref{shearer} applied to the 
$k$-set cover $\cA_k=\{A,\;|A|=k\}$, with $n_\pm(\cA)=\binom{n-1}{k-1}$ one has
$$
\sum_{A\in\cA_k} \ent(f_A) \leq \binom{n-1}{k-1} \ent(f),
$$
for all $k=2,\dots,n$. 
These observations show that 
\begin{equation}\label{lent15}
2^{-n+1}\sum_A \ent(f_A)\leq 2^{-n+1}\ent(f)\sum_{k=2}^n \binom{n-1}{k-1} = (1-2^{-n+1})\ent(f).
\end{equation}
It follows that \eqref{lent5} holds with $\kappa=2^{-n+1}$, a very poor estimate compared with the constant $\k(\nu)=(1-2^{-n+1})/(n-1)$ appearing in Theorem \ref{main2}.

The key ingredient in our analysis is a suitable refinement of Proposition \ref{shearer}.
Various improved versions of Shearer bounds have been proved recently; see \cite{MT,BB}. Our refinements below (see Lemmas~\ref{delstar} and~\ref{delstarg})
appear to be new, and are inspired by the work of Balister and Bollob{\'a}s \cite{BB}, who emphasized the  role played by the sub-modularity property of entropy. Recall that a map $A\mapsto h(A)$, $A\subset[n]$ is called {\em sub-modular} if for all $A,B\subset [n]$
\begin{equation}\label{submo}
h(A)+h(B)\geq h(A\cap B) + h(A\cup B).
\end{equation}
We shall need the following simple lemma.
\begin{lemma}\label{submole}
For every nonnegative function $f$ and any product measure $\mu$, the map 
\begin{equation}\label{submodu}
A\mapsto h(A)=-\ent(f_A)
\end{equation} 
is sub-modular, and $h(\emptyset)=0$. 
\end{lemma}
\begin{proof}
This follows from \eqref{shannon1}, and the fact that  
\begin{equation}\label{shent}
H(Z_A)+H(Z_B)\geq H(Z_{A\cap B}) + H(Z_{A\cup B})\,,
\end{equation}
for all $A,B\subset[n]$, and any random variable $Z=(Z_1,\dots,Z_n)$. The proof of \eqref{shent} is standard; see, e.g., \cite{BB}. 
\end{proof}
We now proceed with the proof of Theorem \ref{main2}, beginning with the case of uniform crossover. 
\subsubsection{Uniform crossover}
Our analysis hinges on the following improved Shearer bound, which sharpens
inequality~\eqref{lent10} in Proposition~\ref{shearer}.
\begin{lemma}\label{delstar}
For every $n\geq 2$, for any sub-modular map $h$ with $h(\emptyset)=0$,
\begin{equation}\label{submo*}
\sum_{A\subset [n]} h(A)\geq \frac{(n-2)2^{n-1}+1}{n-1}\,h([n])\,+\,\frac{2^{n-1}-1}{n-1}\sum_{i=1}^n h(\{i\})\,.
\end{equation}
\end{lemma}
\begin{proof}
For any $1\leq k\leq n$,  define  $\cA_k = \{A\subset[n]:\, |A|=k\}$, and  
$$
\varphi_k=\sum_{A\in\cA_k} h(A).$$ 
Let us show that, for all $2\leq k\leq n-1$,
\begin{equation}
\label{bb3}
\varphi_k\geq \frac1k\binom{n}{k-1}\varphi_n + \frac{n-k}{k}\varphi_{k-1}\,.
\end{equation}
Write $$\varphi_k= 
\frac1k\sum_{A'\in\cA_{k-1}}\sumtwo{A\in\cA_k:}{A'\subset A}h(A).
$$
Let us first prove that, for any $A'\in\cA_{k-1}$, one has 
\begin{equation}\label{bb2}
\sumtwo{A\in\cA_k:}{A'\subset A}h(A)\geq h([n]) + (n-k) h(A') .
\end{equation}
Assume, without loss of generality, that $A'=\{1,\dots,k-1\}$. The sum above then becomes
$$
\sumtwo{A\in\cA_k:}{A'\subset A}h(A)=\sum_{j=k}^{n}h(A'\cup\{j\}).
$$
From sub-modularity \eqref{submo}, one has $$h({A'\cup\{k\}}) + h({A'\cup\{k+1\}})\geq h({A'}) + 
h({A'\cup\{k,k+1\}}).
$$ Set $U_\ell=A'\cup\{k,\ldots,k+\ell\}$, for $\ell\in\{0,\dots,n-k\}$. Then, recursively:
$$
\sum_{j=k}^{k+\ell} h({A'\cup\{j\}}) \geq \ell\, h({A'}) + h({U_\ell}). 
$$
Setting $\ell=n-k$ proves \eqref{bb2}, since ${U_{n-k}}=[n]$. 
Next, using \eqref{bb2} and noting that $h([n])=\varphi_n$, one has 
\eqref{bb3}.

Iterating \eqref{bb3}, we arrive at
\begin{equation}
\label{bb4}
\varphi_k\geq c(k,n)
\varphi_n+ d(k,n)
\,\varphi_{1}\,,
\end{equation}
where
\begin{gather*}
d(k,n)=\frac{(n-2)!}{k!(n-k-1)!},\\
c(k,n)=\left[\frac1k\binom{n}{k-1} + \frac{(n-k)}{k(k-1)}
\binom{n}{k-2}+\dots
+\frac{n(n-3)!}{k!(n-k-1)!}\right].
\end{gather*}
Note that \eqref{bb4} holds for all $k=2,\dots,n-1$. It can be extended to $k=1,\dots,n$ as well, provided we set $c(1,n)=0$, $c(n,n)=1$,  $d(n,n)=0$ and $d(1,n)=1$. 
Thus, the claim \eqref{submo*} follows once we prove
\begin{equation}
\label{lab1}
\sum_{k=1}^nc(k,n) =  \frac{(n-2)2^{n-1}+1}{n-1}
\end{equation}
and
\begin{equation}
\label{lab2}
\sum_{k=1}^nd(k,n) =  \frac{2^{n-1}-1}{n-1}.
\end{equation}
The identity \eqref{lab2} follows from 
$$
\sum_{k=1}^{n-1} \frac{(n-2)!}{k!(n-k-1)!} = \frac1{n-1}\sum_{k=1}^{n-1} \binom{n-1}{k}
= \frac{(2^{n-1}-1)}{n-1}.
$$
To check \eqref{lab1}, notice that 
if $h(A)=\ind(A\neq\emptyset)$, then \eqref{bb2} and \eqref{bb3} are both identities. Therefore, \eqref{bb4} is an identity as well, for all $k=1,\dots,n$ (with the above definitions of $c(k,n)$ and $d(k,n)$). Since here $\varphi_k=\binom{n}{k}$, summing over $k=1,\dots,n$ in \eqref{bb4} one has
$$
2^n -1 = \sum_{k=1}^nc(k,n)  + n \,\frac{2^{n-1}-1}{n-1},
$$
which is equivalent to \eqref{lab1}. This ends the proof of \eqref{submo*}.
\end{proof}
From Lemmas \ref{delstar} and \ref{submole},
\begin{equation}\label{submos*}
\sum_{A\subset [n]} \ent(f_A)\leq \frac{(n-2)2^{n-1}+1}{n-1}\,\ent(f)\,+\,\frac{2^{n-1}-1}{n-1}\sum_{i=1}^n \ent(f_i)\,.
\end{equation}
Since $f\in\cS_\mu$, one has $\ent(f_i)=0$ for all $i$, and therefore 
\begin{equation}\label{submoso*}
2^{-n}\sum_{A\subset [n]} \left(\ent(f_A)+\ent(f_{A^c})\right)\leq \frac{(n-2)+2^{-n+1}}{n-1}\,\ent(f)\,,
\end{equation}
which implies the lower bound $\k(\nu)\geq (1-2^{-n+1})/(n-1)$.

To prove the upper bound, we argue as follows. Suppose $X_1=\cdots =X_n=X$ so that $\cX= X^n$, and suppose $\mu$ is a product of identical probability measures $\mu_0$ on $X$. Let $Z_0$ denote a random variable with values in $X$ with 
distribution $\mu_0$, and call $Z$ the random variable $(Z_0,\dots,Z_0)$ with values in $\cX$, i.e., $Z$ consists of $n$ identical copies of $Z_0$. Clearly, \begin{equation}\label{lent109}
H(Z_A)=H(Z_0)=-\sum_{x\in X}\mu_0(x)\log \mu_0(x),\end{equation}
for any $A\subset[n]$, $A\neq\emptyset$.  Next, let $f$ denote the probability density of $Z$ with respect to $\mu$, i.e., $f\mu$ is the law of $Z$. Notice that $f\in\cS_\mu$, since $f_i=1$ for all $i$. 
By \eqref{shannon1} one has, for all $A\subset[n]$,
 \begin{equation}\label{lent110}
\ent(f_A)=(|A|-1)\ind(A\neq\emptyset)H(Z_0).
\end{equation}
It follows that 
$$
\sum_{A\subset [n]} \ent(f_A)= [(n-2)2^{n-1}+1]H(Z_0).$$
Since $\ent(f)=(n-1)H(Z_0)$ we see that this choice of $f$ saturates the bound \eqref{submos*}.
This concludes the proof of Theorem \ref{main2} for uniform crossover.

\subsubsection{The Bernoulli($q$) model} 
We need the following extension of Lemma \ref{delstar}. 
\begin{lemma}\label{delstarg}
For every $\g\in(0,\infty)$, for any sub-modular map $h$ with $h(\emptyset)=0$,
\begin{equation}\label{submog}
\sum_{A\subset [n]} \g^{|A|}h(A)\geq \frac{(1+\g)^{n-1}[\g(n-1)-1]+1}{n-1}\,h([n])\,+\,\frac{(1+\g)^{n-1}-1}{n-1}\sum_{i=1}^n h(\{i\})\,.
\end{equation}
\end{lemma}
\begin{proof}
The left hand side of \eqref{submog} coincides with $\sum_{k=1}^n \hat\varphi_k$, where $\hat\varphi_k:=\g^k\varphi_k$ and $\varphi_k$ was defined in Lemma \ref{delstar}. 
From \eqref{bb4} it follows that
\begin{equation}
\label{bb4g}
\hat\varphi_k\geq \g^{k}c(k,n)
\varphi_n+ \g^{k}d(k,n)
\, \varphi_{1}\,.
\end{equation}
Thus \eqref{submog} will follow if we can prove
\begin{equation}
\label{lab1g}
\sum_{k=1}^n \g^k c(k,n) =  \frac{(1+\g)^{n-1}[\g(n-1)-1]+1}{n-1}
\end{equation}
and
\begin{equation}
\label{lab2g}
\sum_{k=1}^n \g^k d(k,n) =  \frac{(1+\g)^{n-1}-1}{n-1}.
\end{equation}
It is not hard to check that these identities follow in the same way as \eqref{lab1}-\eqref{lab2}.
\end{proof}

Next, observe that for $\nu(A)=q^{|A|}(1-q)^{n-|A|}$, one has 
\begin{align}\label{diac}
&\sum_{A\subset [n]} \nu(A)[\ent(f_A)+\ent(f_{A^c})] \nonumber\\
&\qquad\quad = (1-q)^n\sum_{A\subset [n]}(\tfrac{q}{1-q})^{|A|} \ent(f_A)
+ q^n\sum_{A\subset [n]}(\tfrac{q}{1-q})^{-|A|} \ent(f_A)\,.
\end{align}
We estimate each sum above
using Lemma \ref{delstarg} with $h(A)=-\ent(f_A)$, once with $\g=\tfrac{q}{1-q}$ and once with $\g=(\tfrac{q}{1-q})^{-1}$. This yields
\begin{align}\label{diac2}
\sum_{A\subset [n]} \nu(A)[\ent(f_A)+\ent(f_{A^c})]&
\leq \D(q,n)\,\ent(f),
\end{align}
where 
\begin{align*}
& \D(q,n) = (1-q)^n \,\frac{(1-q)^{1-n}
\left[\tfrac{q}{1-q}(n-1)-1\right]+1}{n-1} \,+ \\ &\;\;\qquad \qquad  + q^n\,\frac{q^{1-n}
\left[\tfrac{1-q}{q}(n-1)-1\right]+1}{n-1} = 1-\frac{1-q^{n}-(1-q)^n}{n-1}.
\end{align*}
This implies the lower bound
$$
\k(\nu)\geq \frac{1-q^{n}-(1-q)^n}{n-1}.
$$
To prove the upper bound, we argue as in \eqref{lent109}-\eqref{lent110}. Using the same function $f$ defined there, one has
\begin{align*}
&\sum_{A\subset [n]} \nu(A)(\ent(f_A)+ \ent(f_{A^c}))\\&
 = \sum_{A\subset [n]} q^{|A|}(1-q)^{n-|A|}\left[(|A|-1)\ind(A\neq\emptyset) + (|A^c|-1)\ind(A\neq[n]) \right]\,H(Z_0)
 \\ & = \left(n-2 + \nu([n]) + \nu(\emptyset)\right)H(Z_0) = (n-1)H(Z_0)\left[1-\tfrac{1-q^{n}-(1-q)^n}{n-1}\right].
\end{align*}
Since $\ent(f)=(n-1)H(Z_0)$, this choice of $f$ saturates the bound  \eqref{diac2}. This proves Theorem \ref{main2} for the Bernoulli($q$) model. 

\subsubsection{Single site recombination}\label{singles}
When $\nu(A) = \tfrac1n\sum_{i=1}^n\ind(A=\{i\})$, using $\ent(f_i)=0$ for all $i\in[n]$,  
\eqref{lent5} becomes
 \begin{equation}\label{lent6}
 \frac1n\sum_{i=1}^n\mu\left[f_{\{i\}^c} \log f_{\{i\}^c}\right]\leq (1-\kappa)\,\ent(f)\,.
\end{equation}
We note that an application of Proposition \ref{shearer} with the $(n-1)$-cover $\cA_{n-1}=\{\{i\}^c,\;i\in[n]\}$ gives  the inequality \eqref{lent6} with constant $\k=1/n$. 

To prove that it can be strengthened to $\k=1/(n-1)$ we observe that the left hand side of \eqref{lent6} can be written
  \begin{equation}\label{lent13}
 \frac1n\sum_{i=1}^n\mu\left[f_{\{i\}^c} \log f_{\{i\}^c}\right] = -\frac1n\,\varphi_{n-1},
\end{equation}
where $\varphi_k$ is defined in the proof of Lemma \ref{delstar} with $h(A)=-\ent(f_A)$.
From \eqref{bb4},
   \begin{equation}\label{lent113}
 \frac1n\sum_{i=1}^n\mu\left[f_{\{i\}^c} \log f_{\{i\}^c}\right] \leq  \frac1n\,c(n-1,n)\,\ent(f).
\end{equation}
The coefficient $c(n-1,n)$ can be computed as in the proof of Lemma \ref{delstar}, and one finds 
$c(n-1,n)=n - \tfrac{n}{n-1}$. 
Thus \eqref{lent6} holds with $\kappa=\tfrac1{n-1}$, which proves the lower bound 
$\k(\nu)\geq \tfrac1{n-1}$.


The upper bound follows as in \eqref{lent109}-\eqref{lent110}. Indeed, with that choice of $f$ and $\mu$ one has  
\begin{equation}\label{lent19}
 \frac1n\sum_{i=1}^n\mu\left[f_{\{i\}^c} \log f_{\{i\}^c}\right] = (n-2)H(Z_0) = \frac{n-2}{n-1}\,\ent(f)\,.
\end{equation}
This concludes the proof of Theorem \ref{main2} for single site recombinations.
\subsubsection{One-point crossover}\label{single}
Here $$\nu(A)=\frac1{n+1}\sum_{i=0}^n\ind(A=J_i),$$ where $J_0=\emptyset$ and $J_i=\{1,\dots,i\}$, $i\geq 1$. 
Fix $f\geq 0$ and define, for $i=1,\dots,n$: $$u(i) = \ent(f_{J_i}),\qquad \bar u(i) = \ent(f_{J_{i-1}^c})\,.$$ 
Notice that $J_{i}\cap J_{i-1}^c= \{i\}$, and $J_{i}\cup J_{i-1}^c=[n]$. Thus, from Lemma \ref{submole} one has
\begin{equation}\label{optg1}
u(i)+\bar u(i)\leq \ent(f_i) + \ent(f)\,.
\end{equation}
Therefore,
\begin{align}\label{diac1}
&\sum_{A\subset [n]} \nu(A)[\ent(f_A)+\ent(f_{A^c})]  = \frac1{n+1}\sum_{i=1}^n(u(i)+\bar u(i))
\nonumber\\
&\qquad\qquad\qquad\leq \frac{n}{n+1}\,\ent(f)+\frac1{n+1}\sum_{i=1}^n\ent(f_i)\,.
\end{align}
If $f\in\cS_\mu$ one has $\ent(f_i)=0$ for all $i\in[n]$, and \eqref{diac1} proves the desired upper bound. To prove that \eqref{diac1} is optimal, notice that with the argument in \eqref{lent109}-\eqref{lent110} one obtains $u(i)+\bar u(i)=  \ent(f)$ for every $i=1,\dots,n$, and therefore \eqref{diac1} is an identity for this choice of $f$. This proves Theorem \ref{main2} for the one-point crossover model, and thus concludes the proof of the theorem.

 \subsection{Proof of Theorem \ref{main}}
From Lemma \ref{le2} and Theorem \ref{main2} we have already obtained the desired lower bounds $\d(\nu)\geq \k(\nu)$.  The upper bounds on $\d(\nu)$ are based on the following estimate.  

\begin{proposition}\label{sharp}
Let $\nu$ be one of the four recombination distributions, and set $\cX=\{0,1\}^n$. 
Let $w=w(n)=2^{-n}$, and let $\mu$ be the product of  
Bernoulli($w$) probability measures. Define 
\begin{equation}\label{delta}
\D_\nu = \sum_{A}\nu(A)\left(2^{-|A|} + 2^{-|A^c|}\right).
\end{equation}
Then, 
$$\d(\mu,\nu)\leq  \frac{4(1- \D_\nu)}n + O(n^{-2}).$$
\end{proposition}
Let us first check that Proposition \ref{sharp} implies the upper bounds on $\d(\nu)$ announced in Theorem \ref{main}. Indeed, 
in the case of single site recombination $\nu(A)=\frac1n\sum_{i=1}^n\ind(A=\{i\})$, one has \begin{equation}\label{deltan1}
\D_\nu = \frac12 + O(2^{-n}),\end{equation}
and therefore $\d(\nu)\leq \d(\mu,\nu)\leq \tfrac2n + O(n^{-2})$.
In the case of one-point crossover, one finds easily that $\D_\nu=O(1/n)$, and therefore $\d(\nu)\leq \tfrac4n + O(n^{-2})$. 
For uniform crossover $\nu(A)= 2^{-n}$, one has that $\D_\nu$ is exponentially small, and thus again $\d(\nu)\leq \tfrac4n + O(n^{-2})$. 
Finally, for the Bernoulli($q$) model, one finds 
\begin{equation}\label{deltan}
\D_\nu = \sum_{k=0}^n\binom{n}{k}q^{k}(1-q)^{n-k}(2^{-k}+2^{-n+k}) = \left(1-\tfrac{q}2\right)^n + \tfrac1{2^n}(1+q)^n.
\end{equation}
Since $q\leq 1/2$, $\D_\nu= \left(1-q/2\right)^n +O\left((3/4)^n\right)$, which yields the claimed upper bound on $\d(\nu)$.

\begin{proof}[Proof of Proposition \ref{sharp}]
Let $B(u)\in\cP(\cX) $ denote the product of independent Bernoulli with parameter $u\in[0,1]$, so that $\mu=B(w)$, and define $f=p/\mu$,  where
\begin{align}\label{pchoice}
p=w^2B(1)+(1-w)^2B(0)+2w(1-w)B(\tfrac12).
\end{align}
It is easily checked that $p$ and $\mu$ have the same marginals, i.e., $p_i=\mu_i$, so that $f_i=1$ for all $i\in[n]$. 
Then $p$ can be written as
$$
p(\si)=a\ind(\si\equiv 1)  + b\ind(\si\equiv 0) + c\ind(\si\not\equiv 1\;\text{and}
\;\si\not\equiv 0),
$$
where $a=w^2+2w(1-w)2^{-n}$, $b=(1-w)^2+2w(1-w)2^{-n}$, and $c=2w(1-w)2^{-n}$.
The relative entropy is given by
\begin{align}\label{compu01}
\ent(f) = a \log \left(\tfrac{a}{w^n}\right) + b\log\left(\tfrac{b}{(1-w)^n}\right) +c\sum_{k=1}^{n-1}
\binom{n}{k}\log\left(\tfrac{c}{w^k(1-w)^{n-k}}\right).
\end{align}
Using $c\sum_{k=1}^{n-1}
\binom{n}{k} = 2w(1-w)(1-2^{-n+1})$ and $c\sum_{k=1}^{n-1}
\binom{n}{k} k = nw(1-w)(1-2^{-n+1})$, 
\begin{align}\label{compu010}
\ent(f) & = a \log \left(\tfrac{a}{w^n}\right) + b\log\left(\tfrac{b}{(1-w)^n}\right) +
\\ & \;\; + 2w(1-w)(1-2^{-n+1})\log\left(\tfrac{c}{(1-w)^{n}}\right) + nw(1-w)(1-2^{-n+1})\log\left(\tfrac{1-w}{w}\right). \nonumber 
\end{align}
Using $w=2^{-n}$,  one finds 
\begin{gather*}
a\log \left(\tfrac{a}{w^n}\right) =O(2^{-n})\,;\qquad b\log\left(\tfrac{b}{(1-w)^n}\right)=n2^{-n} + O(2^{-n})\,;\\
2w\log\left(\tfrac{c}{(1-w)^{n}}\right) = -4n2^{-n}\log 2 + O(2^{-n})\,;\qquad 
nw\log\left(\tfrac{1-w}{w}\right)= n^22^{-n}\log 2+ O(2^{-n}).
\end{gather*}
Therefore,
\begin{align}\label{compu1}
\ent(f) 
=n^22^{-n}\log 2 - 4 n2^{-n}\log 2  + n2^{-n} + O(2^{-n}).
\end{align}
Next, we compute the entropy production $D(f,f)$. 
Define \begin{align}\label{compu0}
\wt p(\si) = \sum_A\nu(A) (p_A\otimes p_{A^c})\,.
\end{align}
As in \eqref{lent2} we write
\begin{align}\label{compu3}
D(f,f) = \ent(f) - 2^{-n}\sum_A\mu[f_Af_{A^c}\log f]= \ent(f)-\wt p\,[\log(p/\mu)]\,.
\end{align}
Observe that
\begin{align*}
&\wt p\,[\log(p/\mu)]  = \a_n \log \left(\tfrac{a}{w^n}\right) + \a_0 \log\left(\tfrac{b}{(1-w)^n}\right) +\sum_{k=1}^{n-1}\a_{k}
\log\left(\tfrac{c}{w^k(1-w)^{n-k}}\right),
\end{align*}
where $\a_k=\sum_{\si:\,|\si|=k} \wt p(\si)$ and $|\si|$ denotes the number of $1$'s in $\si$.
From the conservation of marginals one has that $\sum_{k=1}^nk\a_k=nw$. Therefore,
\begin{align}\label{compu0030}
\wt p\,[\log(p/\mu)]  
& = \a_n \log \left(\tfrac{a}{w^n}\right) + \a_0 \log\left(\tfrac{b}{(1-w)^n}\right)+\nonumber\\
& \qquad  (1-\a_0-\a_n)\log\left(\tfrac{c}{(1-w)^n}\right) + (nw - n\a_n)
\log\left(\tfrac{1-w}{w}\right).
\end{align}
From \eqref{compu010} and \eqref{compu3} we obtain
\begin{align}\label{compu30}
D(f,f) 
& =(a-\a_n) \log \left(\tfrac{a}{w^n}\right) + (b-\a_0) \log\left(\tfrac{b}{(1-w)^n}\right)+\nonumber\\
& \qquad + \b
\log\left(\tfrac{c}{(1-w)^n}\right) + \g
\log\left(\tfrac{1-w}{w}\right),
\end{align}
where the new coefficients $\b,\g$ are given by
$$
\b = 2w(1-w)(1-2^{-n+1}) -(1-\a_0-\a_n) \,,\qquad \g =n\a_n - nw(2^{-n+1}(1-w)+w) .
$$ 
In order to estimate the ratio $D(f,f)/\ent(f)$ we need to control $\a_0$ and $\a_n$. These can be computed as follows. 
For any $A\subset [n]$, $u,v\in[0,1]$, write
\begin{align}\label{bern2}
B(u)_A\otimes B(v)_{A^c} = B(u_Av_{A^c})\,,
\end{align}
where $B(u_Av_{A^c})$ denotes the product of independent Bernoulli measures such that the sites in $A$ have parameter $u$ and the sites in $A^c$ have parameter $v$. 
It follows that
\begin{align*}
\wt p &= \textstyle{\sum_A\nu(A)\big\{w^4 B(1) + (1-w)^4 B(0) + 4w^2(1-w)^2
B(1/2)  + w^2(1-w)^2B(1_A0_{A^c})
}\\ &\qquad \quad +w^2(1-w)^2B(0_A1_{A^c})+2w(1-w)^3B(0_A(1/2)_{A^c}) +2w(1-w)^3B((1/2)_{A}0_{A^c}) \\ &\qquad\qquad\quad +2w^3(1-w)B(1_A(1/2)_{A^c}) + 2w^3(1-w)B((1/2)_{A}1_{A^c})
\big\}\,.
\end{align*}
From this expression it is not hard to check that $$\a_n=O(w^2)\,,\qquad\a_0=1-4w +2w\D_\nu + O(w^2),$$ where
$\D_\nu$ is given by \eqref{delta}. 
Moreover, $\b = -2w + 2w\D_\nu + O(w^2)$ and $\g = O(nw^2)$. Therefore, 
\begin{gather*}
(a-\a_n)\log \left(\tfrac{a}{w^n}\right) =O(2^{-n})\,;\qquad (b-\a_0)\log\left(\tfrac{b}{(1-w)^n}\right)=O(2^{-n})\,;\\
\b\log\left(\tfrac{c}{(1-w)^{n}}\right) = -4(1-\D_\nu)n2^{-n}\log 2 + O(2^{-n})\,;\qquad 
\g\log\left(\tfrac{1-w}{w}\right)= O(2^{-n}).
\end{gather*}
From \eqref{compu1} and \eqref{compu30} it follows that
\begin{align}\label{compu4}
\d(\mu,\nu)\leq \frac{D(f,f)}{\ent(f)}=\frac{4(1-\D_\nu)}n\,+ O(n^{-2}).
\end{align}
This concludes the proof of Proposition~\ref{sharp} and of Theorem~\ref{main}.
\end{proof}

\subsection{Proof of Corollary \ref{corodio}}
Let $\nu$ be one of the four recombination distributions and fix an arbitrary initial state $p\in\cP(\cX)$.
Let $\pi=\otimes_{i=1}^np_i$ be the associated product measure. If $\pi$ fails to satisfy $\pi(\si)>0$ for all $\si\in\cX$ we can always redefine $\cX$ so that this holds. Thus, without loss of generality, we assume that $\pi\in\cP_+(\cX)$, and we can work within the general framework developed in Section \ref{setup} with $\mu=\pi$. From Proposition \ref{entroprop} and the observations in Lemma \ref{le1} and Lemma \ref{le2} we know that
\begin{align}\label{ex1}
H(p_t\tc\pi)\leq e^{-\k(\nu)\,t}H(p\tc \pi)
\end{align}
for all $t\geq 0$, where  $\k(\nu)$ is the constant computed in Theorem \ref{main2}. 
This proves the upper bound \eqref{dente1}. 
We turn to the corresponding statement in discrete time.
The following lemma proves the desired upper bound \eqref{dente}.
Let $p^{(k)}= \Psi[p^{(k-1)}]$, $k\in\bbN$, $p^{(0)}=p$, where 
$$
\Psi[p] = \sum_{A}\nu(A)\,(p_A\otimes p_{A^c}).
$$ 
 \begin{lemma}
\label{leo2}
For any initial state $p\in\cP(\cX)$, 
\begin{equation}\label{lent5d}
H(p^{(k)}\tc\pi)\leq (1-\k(\nu))\,H(p^{(k-1)}\tc \pi)\,,
\end{equation}
for all $k\in\bbN$, where $\k(\nu)$ is the constant in Theorem \ref{main2} and $\pi=\otimes_{i=1}^np_i$.
\end{lemma}
\begin{proof}
Since the map $\Psi$ preserves the marginals, it suffices to show that, for any $p$ with the same marginals as $\pi$, one has
\begin{equation}\label{lent5d1}
H(\Psi[p]\tc\pi)\leq (1-\k(\nu))\,H(p\tc \pi)\,,
\end{equation}
Set $f=p/\pi$ and $f_\Psi= \Psi[p]/\pi$. Now, $ f_\Psi = \sum_{A}\nu(A)f_Af_{A^c}$. Convexity of the
function $x\mapsto x\log x$, $x\geq 0$,
shows that
\begin{equation}\label{convei}
\pi[f_\Psi\log f_\Psi]\leq \sum_{A}\nu(A)\pi[f_Af_{A^c}\log f_Af_{A^c}]= \sum_A\nu(A)[\ent(f_A)+\ent(f_{A^c})].
\end{equation}
Therefore, by Theorem \ref{main2}, 
\begin{equation}\label{lent5d2}
\pi[f_\Psi\log f_\Psi]\leq (1-\k(\nu))\,\pi[f\log f],
\end{equation}
which coincides with \eqref{lent5d1}. 
\end{proof}
To conclude the proof of Corollary \ref{corodio} we need to check the bounds in \eqref{dente3}.
To this end, take $\cX=\{0,1\}^n$ and $p$ the probability measure defined in \eqref{pchoice}. If $f=p/\pi$, then  the proof of Proposition \ref{sharp} shows that 
\begin{equation}\label{varprino}
D(f,f)\leq \g(\nu)\,\ent(f)\,,\quad \g(\nu):=\left(\tfrac4n(1-\D_\nu)+O(n^{-2})\right).
\end{equation}
By Proposition \ref{genstat} part 1, one has 
$$
\at{\dert H(p_t\tc\pi)}{t=0^+} = - D(f,f)\geq - \g(\nu)\,H(p\tc\pi).
$$
On the other hand, in discrete time one has 
$$
H(p^{(1)}\tc\pi) = \pi[f_\Psi\log f_\Psi]. 
$$
From the variational principle for relative entropy,
\begin{equation}\label{varprin}
\pi[f_\Psi\log f_\Psi]\geq \pi[f_\Psi\log f] = \sum_{A}\nu(A)\pi[f_A f_{A^c}\log f]\,.
\end{equation}
From \eqref{lent5}, \eqref{varprin} and \eqref{varprino} it follows that 
$$
H(p^{(1)}\tc\pi) \geq   H(p\tc\pi) - D(f,f)\geq \left(1- \g(\nu)\right)\,H(p\tc \pi)\,,
$$
which proves \eqref{dente3}. Finally, the estimate  
\begin{equation}\label{prin}
\g(\nu) \leq C\k(\nu)
\end{equation}
follows easily from \eqref{deltan1}-\eqref{deltan}. Because of the $O(n^{-2})$ correction in $\g(\nu)$ one needs to require $q\geq n^{-2}$ in the case of the Bernoulli($q$) model. 
\bibliographystyle{plain}

\bibliography{recombinations}

\end{document}